\RequirePackage{snapshot}
\documentclass[11pt]{article}
\usepackage{microtype}
\usepackage{units}
\usepackage{amsmath}
\usepackage{amsthm}
\usepackage{amssymb}
\usepackage{mathtools}
\usepackage[margin=1in]{geometry}
\usepackage{color}
\usepackage{url}
\usepackage{multicol}
\usepackage[hidelinks]{hyperref}
\usepackage{graphicx}
\usepackage{caption}
\usepackage{subcaption}
\usepackage{booktabs}
\usepackage{pgfplots}\pgfplotsset{compat=1.4}
\usepackage{subdepth}
\usepackage[sort,numbers,longnamesfirst]{natbib}

\newcommand{\inpar}[1]{ \left( #1 \right) } %

\newcommand{\bigo}[1]{\cO\!\left( #1 \right)}

\newcommand{\st}{\mathop{\hbox{subject to}}}
\newcommand{\gam}{\gamma}
\newcommand{\Diag}{\mbox{\rm Diag}\,}

\newcommand{\lam}{\lambda}

\newcommand{\di}{\hbox{dist}}
\newcommand{\sig}{\sigma}
\newcommand{\alf}{\alpha}
\newcommand{\del}{\delta}

\newcommand{\vp}{\varphi}
\newcommand{\vpen}{\vp_{en}}

\newcommand{\epi}{\mathrm{epi}\,}

\newcommand{\norm}[1]{\| #1 \|} %

\newcommand{\tnorm}[1]{\norm{#1}_2}
\newcommand{\inorm}[1]{\norm{#1}_\infty}
\newcommand{\onorm}[1]{\norm{#1}_1}

\newcommand{\ip}[2]{\left\langle #1,\, #2\right\rangle}
\newcommand{\set}[2]{\left\{#1\,\left|\ #2\right.\right\}}

\newcommand{\submax}{_{\max}}
\newcommand{\submin}{_{\min}}
\newcommand{\dom}[1]{\mathrm{dom}\,#1}

\newcommand{\cl}[1]{\mathrm{cl}\left(#1\right)}
\newcommand{\argmin}{\mathop{\mathrm{argmin}}}

\newcommand{\ball}{\bB}
\newcommand{\e}{\mathrm{e}}

\newcommand{\range}{\mathrm{range}\,}

\newcommand{\half}{\tfrac{1}{2}}

\newcommand{\abs}[1]{ \left| #1 \right| } %
\newcommand{\nucnorm}[1]{ \left| \left| #1 \right| \right|_*} %
\newcommand{\twonorm}[1]{ \left| \left| #1 \right| \right|_2 } %

\newcommand{\sign}{\mathsf{sgn}}

\DeclareMathOperator{\interior}{int}

\newcommand{\extR}{\overline\bR} %
\newcommand{\bR}{\mathbb{R}}

\newcommand{\Rn}{\bR^n}

\newcommand{\bB}{\mathbb{B}}
\newcommand{\cP}{\mathcal{P}}
\newcommand{\cQ}{\mathcal{Q}}

\newcommand{\cA}{{\mathcal{A}}}
\newcommand{\cD}{{\mathcal{D}}}

\newcommand{\cK}{{\mathcal{K}}}
\newcommand{\cS}{\mathcal{S}}

\newcommand{\cO}{\mathcal{O}}

\newcommand{\tr}{\mathrm{tr}\,}
\newcommand{\cC}{\mathcal{C}}

\newcommand{\cX}{\mathcal{X}}

\newcommand{\OPT}{\hbox{OPT}}
\newcommand{\chat}{\skew3\widehat c}
\newcommand{\yhat}{\skew2\widehat y}

\newcommand{\textt}[1]{\hbox{\qquad#1\qquad}}

\def\minim{\mathop{\hbox{minimize}}}
\def\maxim{\mathop{\hbox{maximize}}}
\def\minimize#1{\displaystyle\minim_{#1}}
\def\maximize#1{\displaystyle\maxim_{#1}}
\newcommand{\T}{^T\!}

\newcommand{\QPlam}{QP$_{\lambda}$}
\newcommand{\LStau}{LS$_{\tau}$}
\newcommand{\BPsig}{BP$_{\sigma}$}
\newcommand{\Psig}{\ensuremath{\cP_\sigma}}
\newcommand{\Qtau}{\ensuremath{\cQ_\tau}}

\theoremstyle{plain} %
\newtheorem{theorem}{Theorem}[section]
\newtheorem{lemma}[theorem]{Lemma}

\newtheorem{proposition}[theorem]{Proposition}

\theoremstyle{definition}
\newtheorem{defn}[theorem]{Definition}

\newtheorem{rem}[theorem]{Remark}

\theoremstyle{definition} %
\newtheorem{exa}[theorem]{Example}

\numberwithin{equation}{section}

\usepackage[vlined,ruled]{algorithm2e}
\SetKwComment{Comment}{[}{]}

\title{Level-set methods for convex optimization}

\author{Aleksandr Y. Aravkin\thanks{Department of Applied Mathematics,
    University of Washington, Seattle, WA 98195, USA;
    \url{sites.google.com/site/saravkin/}; Research supported by the
    Washington Research Foundation Data Science Professorship.}  \and
  James V. Burke\thanks{Department of Mathematics, University of
    Washington, Seattle, WA 98195, USA;
    \url{www.math.washington.edu/\~burke/ };
    Research supported in part by the NSF award DMS-1514559.}\and Dmitriy
  Drusvyatskiy\thanks{Department of Mathematics, University of
    Washington, Seattle, WA 98195, USA;
    \url{www.math.washington.edu/\~ddrusv/}; Research supported by the
    AFOSR YIP award FA9550-15-1-0237.} \and Michael
  P. Friedlander\thanks{Department of Mathematics, UC Davis, One
    Shields Ave, Davis, CA 95616; \url{www.math.ucdavis.edu/\~mpf/ };
    Research supported by the ONR award N00014-16-1-2242.}
  \and Scott Roy\thanks{Department of Mathematics, University of
    Washington, Seattle, WA 98195, USA; Research supported in part by
    the AFOSR YIP award FA9550-15-1-0237.}}

\begin{document}

\date{February 3, 2016}

\maketitle

\begin{quote}
  \begin{center}\bf Abstract \end{center}
  Convex optimization problems arising in applications often have
  favorable objective functions and complicated constraints, thereby precluding 
  first-order methods from being immediately applicable. We
  describe an approach that exchanges the roles of the objective and
  constraint functions, and instead approximately solves a sequence of
  parametric level-set problems.
  A zero-finding procedure, based on inexact function evaluations and
  possibly inexact derivative information, leads to an efficient
  solution scheme for the original problem.  We describe the
  theoretical and practical properties of this approach for a broad
  range of problems, including low-rank semidefinite optimization, sparse optimization, and generalized linear models 
  for inference.
\end{quote}

\clearpage
\tableofcontents

\clearpage
\section{Introduction}
\label{sec:intro}
To motivate the discussion, consider the typical problem 
of recovering a sparse vector $x$ that
approximately satisfies the linear system $Ax=b$. This
task often arises in applications, such as compressed sensing and model
selection.  Standard approaches, based on convex optimization, rely on
solving one of the following problem formulations.
\begin{center}
\begin{tabular}{c@{\qquad}c@{\qquad}c}
\toprule
 \BPsig & \LStau & \QPlam
 \\\midrule
  $\begin{array}[t]{cl}
        \displaystyle\min_{x} & \|x\|_{1}
       \\\textrm{s.t.} & \frac{1}{2}\|Ax-b\|^2_{2}\le\sigma
         \end{array}$
&$\begin{array}[t]{cl}
         \displaystyle\min_{x} & \frac{1}{2}\|Ax-b\|^2_{2}
       \\\textrm{s.t.} & \|x\|_{1}\le\tau
         \end{array}$
 & $\begin{array}[t]{cl}
         \displaystyle\min_{x} & \frac{1}{2}\|Ax-b\|^2_2+\lambda \|x\|_1
         \end{array}$
\\\bottomrule
\end{tabular}
\end{center}

\noindent Computationally, \BPsig\ is perceived to be the most
challenging of the three because of the complicated geometry of the
feasible region. For example, a projected- or proximal-gradient method
for \LStau\ or \QPlam\ requires relatively little cost per
iteration\footnote{Projection onto the ball $\{x:\norm{x}_1\le\tau\}$
  requires $\bigo{n\log n}$ operations; the proximal map for the
  function $\lambda\norm{x}_1$ requires $\bigo{n}$ operations.} beyond
forming the product $Ax$ or $A\T y$. In contrast, a comparable
first-order method for \BPsig, such as the alternating direction
method of multipliers
(ADMM)~\cite{GabayMercier:76,boyd2011distributed}, requires at each
iteration the solution of a linear least-squares
problem~\cite{bioucas2010alternating} and maintains iterates that are
both infeasible and suboptimal.  Consequently, problems \LStau\ and
\QPlam\ are most often solved in practice, and most algorithm
development and implementation targets these versions of the problem.
Nevertheless, the formulation \BPsig\ is often more natural, since the
parameter $\sigma$ plays an entirely transparent role, signifying an
acceptable tolerance on the data misfit.

This paper targets optimization problems generalizing the formulation \BPsig. Setting the stage, consider the pair of problems 
\begin{equation}
\label{eqn:ps}
\minimize{x\in\cX} \quad \varphi(x)
\quad\st\quad\rho(Ax-b)\leq \sigma,
\tag{\Psig} 
\end{equation}
and 
\begin{equation}
\label{eqn:sp}
\minimize{x\in\cX} \quad \rho(Ax-b)
\quad\st\quad\varphi(x)\leq \tau
\tag{\Qtau},
\end{equation}
where $\cX$ is a closed convex set,
$\varphi$ and
$\rho$ are (possibly infinite-valued) closed convex functions, and
$A$ is a linear map. Here, \Psig\ and \Qtau\ extend the problems \BPsig\ and \LStau\,, respectively. 
Such formulations are ubiquitous in contemporary optimization and its applications. %
Our working assumption is that 
the level-set problem \ref{eqn:sp} is easier to solve than
\ref{eqn:ps}---perhaps because it allows for a specialized algorithm
for its solution. In \S\ref{sec:pr_class}, we discuss a range of problems,
including nonsmooth regularization, conic optimization, and
generalized linear models, with this property.

Our main goal is to develop a practical and theoretically sound
algorithmic framework that can be used to harness existing algorithms
for \Qtau\ to efficiently solve the \Psig\ formulation.  As a
consequence, we make explicit the fact that in typical circumstances both problems are
essentially equivalent from the viewpoint of computational
complexity. Hence, there is no reason to avoid any one formulation
based on computational considerations alone.  This observation is very
significant in applications since, although the formulations \Psig\
and \Qtau\ as well as their Lagrangian (or penalty) formulation are,
in a sense, mathematically and computationally equivalent, they are
far from equivalent from a modeling perspective. To illustrate this
point, consider a scenario where we wish to compare the performance of
various regularizers $\varphi_j,\,j=1,\dots,k,$ for a range of values
of the model misfit $\rho(Ax-b)\le \sig_i,\,i=1,\dots,p$.  This is an
important task in machine learning applications where one wishes to
build a classifier based on training data.  In this scenario, the model
formulation \Psig\ is the only one that allows an apples-to-apples
comparison between regularizers $\varphi_i$ for a fixed level of
model misfit. We illustrate this point in \S\ref{sec:learning} on a regularized logistic regression problem.

\subsection{Approach}

The proposed approach, which we will formalize shortly, approximately solves \Psig\ in the sense that it generates a point
$x\in \mathcal{X}$ that is \emph{super-optimal} and
\emph{$\epsilon$-feasible}: 
\[
  \varphi(x)\leq \text{OPT} \textt{and} \rho(Ax-b)\leq \sigma+\epsilon,
\]
where OPT is the optimal value of \Psig. This terminology is used by Harchaoui, Juditsky, and Nemirovski~\cite{cond_grad}, and we adopt it here.
 The proposed strategy is based on
exchanging the roles of the objective and constraint functions in
\Psig, and approximately solving a sequence of level-set problems
\ref{eqn:sp} for varying parameters $\tau$.

How does one use approximate solutions of \Qtau\ to obtain a
super-optimal and $\epsilon$-feasible solution of \Psig, the target
problem?  We answer this by recasting the problem in terms of the value function
for \Qtau:
\begin{equation}\label{eq:value-func}
  v(\tau):=\min_{x\in\cX}\set{\rho(Ax-b)}{\varphi(x)\leq \tau}\ .
\end{equation}
The univariate function $v$ thus defined is nonincreasing and convex
	\cite[Theorem~5.3]{RTR}.	
Under the mild assumption that 
the constraint $\rho(Ax-b)\le\sigma$  is active
at any optimal solution of \ref{eqn:ps},
it is easy to see that
the value $\tau_*:=\mbox{OPT}$ satisfies the equation 
\begin{equation}\label{eqn:root}
v(\tau)=\sigma.
\end{equation}
Conversely, it is immediate that for any $\tau \le \tau_* $ satisfying 
$v(\tau) \le \sigma + \epsilon$,
solutions of \Qtau\ are super-optimal and $\epsilon$-feasible for
\Psig, as required.
In summary, we have translated the problem \Psig\ to that of
finding the minimal root of the nonlinear univariate equation \eqref{eqn:root}.
We show in \S\ref{sec:alg_fram} how approximate solutions of \Qtau\ can serve
as the basis of a root-finding procedure for this key equation. 
For more details about the relationship between \Psig, \Qtau, and their value functions, see Aravkin, Burke, and Friedlander
\cite[Theorem~2.1]{AravkinBurkeFriedlander:2013}.

Our technical assumptions on the problem \Psig\ are relatively few,
and so in principle the approach applies to a wide class of convex
optimization problems. In order to make this scheme practical,
however, it is essential that approximate solutions of
\Qtau\ 
 can be efficiently computed over a sequence of parameters $\tau$.
Hence, efficient implementations attempt to warm start
each new problem. It is thus desirable that the sequence of parameters
$\tau_{k}$ increases monotonically, since this guarantees that the
approximate solutions of $\mathcal{Q}_{\tau_k}$ are feasible for the next
problem in the sequence.  Bisection methods do not have this property,
and we therefore propose variants of secant and Newton methods that
accommodate inexact oracles for $v$ and exhibit the desired
monotonicity property. We prove that the resulting root-finding
procedures unconditionally have a global linear rate of convergence.
Coupled with an evaluation oracle for $v$ that has a cost that is
sublinear in $\epsilon$, we obtain an algorithm with an overall
cost that is also sublinear in $\epsilon$ (modulo a logarithmic
factor).

The outline of the manuscript is as follows. In \S\ref{sec:alg_fram},
we prove complexity bounds and convergence guarantees for the
level-set scheme. We note that the iteration bounds for the root
finding schemes are independent of the slope of $v$ at the root. This
implies that the proposed method is insensitive to the ``width'' of
the feasible region in \Psig.  Such methods are well-suited for
problems \Psig\ for which the Slater constraint qualification fails or
is close to failing; see Example~\ref{ex:edcp}. In \S\ref{sec:refinement},
we consider refinements to the overall method, focusing on linear
least-squares constraints and recovering feasibility.
Section~\ref{sec:pr_class} explores level-set methods in notable
optimization domains, including semi-definite programming, gauge optimization, regularized regression, and generalized linear models.
In \S\ref{case:stud}, we describe the specific steps needed to
implement the root-finding approach for some representative
applications, including low-rank matrix completion 
\cite{RecFazPar:10,ling:2015}, sensor-network localization
\cite{bis_ye_toh_wang,biswas2004semidefinite,biswas2006semidefinite},
and group detection via the elastic net~\cite{EN_2005}.

\subsection{Related work}
\label{sec:related-work}

The intuition behind the proposed framework has a distinguished
history, appearing even in antiquity. Perhaps the earliest instance is Queen Dido's problem and the
fabled origins of Carthage~\cite[Page 548]{EM2001}.  In short, the
problem is to find the maximum area that can be enclosed by an arc of
fixed length and a given line. The converse problem is to find an arc
of least length that traps a fixed area between a line and the
arc. Although these two problems reverse the objective and the constraint, 
the solution in each case is a
semi-circle. The interchange of constraint and
objective provides the foundation for the Markowitz mean-variance
portfolio theory~\cite{Mark1987}; the basic problem is to choose a
portfolio of financial instruments having a lower-bounded rate of
return that minimizes the volatility (variance) of the portfolio. The
converse problem is to maximize the rate of return with a bound on
volatility. Numerous other examples occur throughout history, and the
great variety of possible modern applications is formalized by the
inverse function theorem in Aravkin et al.~\cite[Theorem
2.1]{AravkinBurkeFriedlander:2013}.  More generally, the underlying
idea of the trade-offs between various objectives form the foundations
for multi-objective optimization~\cite{Mie99}.

In the context of numerical optimization, our work is motivated by the
widely-used SPGL1 algorithm
\cite{BergFriedlander:2011,BergFriedlander:2008} for the 1-norm
regularized least-squares problem and its extensions
\cite{AravkinBurkeFriedlander:2013}.  A shortcoming of the numerical
theory to date is the absence of practical complexity and convergence
guarantees.  In this work, we $(i)$ take a fresh new look at this
general framework, $(ii)$ provide rigorous convergence guarantees,
$(iii)$ further illustrate the vast applicability of the approach, and
$(iv)$ show how the proposed framework can be instantiated in concrete
circumstances.

Related ideas appear in Lemar\'echal, Nemirovskii, and Nesterov
\cite{preamb}, who develop their \emph{level} and \emph{truncated
  level} methods using bundle ideas for convex optimization
\cite{Lem75,Wol75}.  Their algorithm is similar in spirit since they
work with lower-level sets of the objective function.  They consider
the convex optimization problem
\[
  \minimize{x\in \mathcal{X}}
  \quad f_0(x)
  \quad\st\quad \mbox{$f_j(x)\le 0$ for $j=1,\dots,m$},
\]
where each function $f_j$ is convex and $\mathcal{X}$ is a nonempty
closed convex set.  The authors define the function
\[
g(\tau):=\min_{x\in \mathcal{X}}\ \max\ \{f_0(x)-\tau,\, f_1(x),\dots,\, f_m(x)\}.
\]
Their algorithm constructs the smallest solution $\tau_*$ to the
equation $g(\tau)=0$; then $\tau_*$ is the optimal value of the
original convex program. See also Nesterov~\cite[\S3.3.4]{Nes04b} for
a discussion.

More recently, Harchoui et al.~\cite{cond_grad}, in a paper inspired
by Lemar\'echal et al.~\cite{preamb}, present an algorithm focusing on
instances of the problem \Psig, where $\rho$ is smooth and $\varphi$
is a gauge of the intersection of a unit ball for a norm and a closed convex
cone. Their zero-finding method is coupled with the Frank-Wolfe
algorithm for generating lower bounds and affine minorants on the
value function.  In contrast, our root finding phase is agnostic to
the inner evaluation algorithm, as is the case in the approaches
described by Aravkin et al.~\cite{AravkinBurkeFriedlander:2013} and
van den Berg and
Friedlander~\cite{BergFriedlander:2011,BergFriedlander:2008}. Consequently,
we see that affine minorants are naturally obtained from dual
certificates in full generality. This is in particular the case for the affine minorants
derived from the Frank-Wolfe algorithm; see~\S\ref{sec:lower minorants via
  duality}.  This observation immediately opens the door to the use of other
primal-dual algorithms, and more generally, to algorithms for solving
the primal and dual problems in parallel. %

\subsection{Notation}
\label{sec:notation}
The notation we use is standard, and follows closely that in
Rockafellar's monograph \cite{RTR}.  The functions we consider take
values in the extended real line $\extR := \bR\cup\{+\infty\}$. For
any function $f\colon\Rn\to\extR$, we use the symbol
$[f \le \alpha] := \{x\in\Rn\colon f(x)\le\alpha\}$ to denote the
$\alpha$-sublevel set. The {\em domain} and the {\em epigraph} of $f$
are defined by
\begin{align*}
\dom f&:=\{x\in\Rn \colon f(x)< +\infty\}\quad\quad \textrm{ and }\quad\quad
\epi f:=\{(x,r)\in\Rn\times \mathbb{R}\colon r\geq f(x)\},
\end{align*}
respectively. We say that $f$ is closed if its epigraph $\epi f$ is a
closed set.  An {\em affine minorant} of $f$ is any affine function
$g$ satisfying $g(x)\leq f(x)$ for all $x$.  The {\em subdifferential}
of a convex function $f\colon\Rn\to\extR$ at a point $x\in \dom f$ is
the set
\[
\partial f(x):=
\set{v\in\bR^n}{f(y)\ge f(x)+\ip{v}{y-x} \mbox{ for all } y\in\bR^n}.
\]
The {\em Fenchel conjugate} of $f$ is the closed, convex function
$f^{\star}\colon\Rn\to\extR$ defined by
\[
 f^{\star}(y):=\sup_x \,\left\{\ip{x}{y}-f(x)\right\}.
\]
The subdifferential and the conjugate of a convex function $f$ are
related by the {\em Fenchel-Young inequality}: any two points $x$ and
$y$ satisfy the inequality
$$f(x)+f^{\star}(y)\geq \langle y,x\rangle.$$
Moreover, equality holds if and only if $y\in \partial f(x)$.
For any set $\cC$ in $\Rn$, we define the associated indicator
function
\[
 \delta_\cC(x) =
 \begin{cases}
    0      & \hbox{if $x\in\cC$},
 \\+\infty & \hbox{otherwise.}
 \end{cases}
\]
The conjugate of the indicator function is simply the support function
$\delta^{\star}_\cC(y) = \sup_{x\in\cC}\,\ip{x}{y}$. In particular,
for any norm $\norm{\cdot}$, the support function of the unit ball
$\{x \colon \norm{x}\leq 1\}$ is the dual norm. The $p$-norms and
corresponding closed unit balls are denoted by $\|\cdot\|_p$ and
$\bB_{p}$, respectively. For any convex cone $\cK$, the {\em dual cone}
is defined by
\[
  \cK^*:=\set{y}{\langle x,y\rangle\geq 0\mbox{ for all } x\in \cK}.
\]

We always endow the Euclidean space of real $m\times n$ matrices
$\mathbb{R}^{m\times n}$ with the trace product
$\langle X,Y\rangle:=\tr(X^TY)$ and the induced Frobenius norm
$\|X\|_F:=\sqrt{\langle X,X \rangle}$. For any matrix
$X\in \mathbb{R}^{m\times n}$, the symbols
$\sigma_1(X)\geq \sigma_2(X)\geq \cdots\geq \sigma_{\min\{m,n\}}(X)$
denote the singular values of $X$.  The Euclidean space of real
$n\times n$ symmetric matrices, written as $\mathcal{S}^n$, inherits
the trace product $\langle X,Y\rangle:=\tr(XY)$ and the corresponding
norm. For any symmetric matrix $X\in \mathcal{S}^{n}$, the symbols
$\lambda_1(X)\geq \lambda_2(X)\geq \cdots\geq \lambda_{n}(X)$ denote
the eigenvalues of $X$.  The closed, convex cone of $n\times n$
positive semi-definite matrices is denoted by
$\mathcal{S}^n_+=\{X\in \mathcal{S}^n: X\succeq 0\}$.
Both the nonnegative orthant $\Rn_+$ and the positive semi-definite
cone $\mathcal{S}^n_+$ are self-dual. The symbol $e\in \Rn$ denotes
the vector of all ones.

\section{Root-finding with inexact oracles}
\label{sec:alg_fram}

Approximate solutions of \Qtau\ are central to our algorithmic
framework, since this is the oracle through which we access $v$. The
available algorithms for \Qtau\ dictate the quality of the oracle. In
this section, we describe the complexity guarantees associated with
two types of oracles: an inexact-evaluation oracle that provides upper
and lower bounds on $v(\tau)$, and an affine minorant oracle that
additionally provides a global linear underestimator on $v$.  The
algorithms presented here apply to any convex nonincreasing function
$f:\bR_+\to\bR$ for which the equation $f(\tau)=0$ has a solution.  In
the following discussion, $\tau_*$ denotes a minimal root of
$f(\tau)=0$.  Given a tolerance $\epsilon >0$, the algorithms we
discuss yield a point $\tau\leq \tau_*$ satisfying
$0\leq f(\tau)\le\epsilon$.

\subsection{Inexact secant}\label{subsec:inexact}

Our first root-finding algorithm is an inexact secant method, and is
based on an oracle that provides upper and lower bounds on the value
$f(\tau)$. 

\begin{defn}[Inexact evaluation oracle]\label{defn:oracle}
  For a function $f:\bR_+\to\bR$, an \emph{inexact evaluation oracle}
  is a map $\cO_f$ that assigns to each pair
  $(\tau,\alpha)\in[f > 0]\times [1,\infty)$ real numbers $(\ell,u)$
  such that $0 < \ell\leq f(\tau)\leq u$ and $u/\ell \leq \alpha$.
\end{defn}

Note that this oracle guarantees a relative accuracy
$u/\ell\leq \alpha$, rather than one based on the absolute gap
$u-\ell$.  This allows the oracle to be increasingly inexact (and
presumably cheaper) for larger values of $f(\tau)$. 
The relative-accuracy condition is no less general than one
based on an absolute gap. In particular, it is readily verified that
for any numbers $l, u$ that satisfy
$0\leq \ell \leq f(\tau)\leq u$ and
$u-\ell\leq (1-1/\alpha)\epsilon$, either
\begin{itemize}
\item $\tau$ is an $\epsilon$-approximate root, i.e.,
  $f(\tau)\leq \epsilon$; or
\item the relative-accuracy condition $1\leq u/\ell\leq \alpha$ is valid.
\end{itemize}
Indeed, provided $f(\tau)> \epsilon$, we deduce
$u/\ell\leq 1+(1-1/\alpha)\epsilon/\ell\leq 1+(1-1/\alpha)u/\ell$,
which after rearranging terms yields the desired inequality
$u/\ell\leq \alpha$.  Hence, the cost of evaluating $f(\tau)$ within
an additive error directly translates into a cost of the same order
for evaluating $f(\tau)$ up to relative accuracy.
Algorithm~\ref{algo:secant} outlines a secant method based on the
inexact evaluation oracle. Theorem~\ref{thm:lin_conv_inex_imp}
establishes the corresponding global convergence guarantees; the proof appears in Appendix~\ref{appendix:proofs}.

\begin{algorithm}[t]
  \DontPrintSemicolon \KwData{A decreasing convex function
    $f:\bR_+\to\bR$ via an inexact evaluation oracle $\mathcal{O}_f$;
    target accuracy $\epsilon >0$; initial points $\tau_0,\tau_1$ with
    $0\le \tau_0<\tau_1$ such that $f(\tau_0)\geq f(\tau_1) > 0$;
    constant $\alpha\in (1,2)$.}

  $(\ell_0,u_0) \gets \mathcal{O}_f(\tau_0,\alpha)$\;
  $k\gets1$\;

  \While{$u_{k} > \epsilon$}{
    $(\ell_{k},u_{k})\gets\mathcal{O}_f(\tau_{k},\alpha)$
      \Comment*{oracle evaluation for lower/upper bounds}
    $u_k\gets \min\{u_k,u_{k-1}\}$
      \Comment*{ensure upper bound decreases}
    $s_k \gets (u_{k-1}- \ell_k)/(\tau_{k-1}-\tau_k)$
      \Comment*{slope of linear approximation}
    $\tau_{k+1} \gets \tau_k - \ell_k/s_k$
      \Comment*{secant iteration}
    $k\gets k+1$\;
  }
  \Return $\tau_{k}$\;
  \caption{Inexact secant method}
  \label{algo:secant}
\end{algorithm}

\begin{theorem}[Linear convergence of the inexact secant
  method]\label{thm:lin_conv_inex_imp}
  The inexact secant method (Algorithm~\ref{algo:secant}) terminates
  after at most
  \[
    k\leq \max \left\{ 2 + \log_{2/\alpha}\!
      \left( 2 C/\epsilon \right),\, 3 \right\}
  \]
  iterations, where $C := \max\{ |s_1|(\tau_*-\tau_1),\, \ell_1 \}$
  and $s_1:=(u_{0}- \ell_1)/(\tau_{0}-\tau_1)$.
\end{theorem}

The iteration bound of the inexact secant method is indifferent to the
slope of the function $f$ at the minimal root $\tau_*$ because
termination depends on function values rather than proximity to
$\tau_*$.  The plots in Figure~\ref{fig:sec_degen} illustrate this
behavior: panel (a) shows the iterates for $f_1(\tau)=(\tau-1)^2-10$,
which has a nonzero slope at the minimal root
$\tau_*=1-\sqrt{10}\approx-2.2$ and so has a non-degenerate solution;
panel (c) shows the iterates for $f_2(\tau) = \tau^2$, which is
clearly degenerate at the solution. The algorithm behaves similarly on
both problems. When applied to the value function $v$ to find a root
of~\eqref{eqn:root}, the algorithm's indifference to degeneracy
translates to an insensitivity to the ``width''~\cite{ren_cond} of the
feasible region of \Psig\,---an unsurprising consequence of the fact
that the scheme maintains infeasible iterates for $\Psig$. Thus such
methods are well-suited for problems \Psig\ for which the Slater
constraint qualification is close to failing. On the other hand, for
non-degenerate problems, we can hope for superlinear convergence when
the function is evaluated with sufficient accuracy (see Theorem
\ref{Superlinear}).

Observe that the iteration bound in
Theorem~\ref{thm:lin_conv_inex_imp} is infinite for $\alpha\geq 2$.
Surprisingly, this is not an artifact of the proof. As illustrated by
Figure~\ref{fig:sec_degen}(b), the inexact secant method behaves
poorly for $\alpha$ close to $2$. Indeed, it can fail to converge linearly (or at all) to
the minimal root for any $\alpha \ge 2$, as the following example
shows. Consider the linear function $f(\tau)=-\tau$ with lower and
upper bounds $\ell_k:=-2\tau_k/(1+\alpha)$ and
$u_k:=-2\alpha \tau_k/(1+\alpha)$. A quick computation shows that the
quotients $q_k:=\tau_k/\tau_{k-1}$ of the iterates satisfy the
recurrence relation $q_{k+1}=(1-\alpha)/(q_k-\alpha)$. It is then
immediate that for all $\alpha \geq 2$, the quotients $q_k$ tend to
one, indicating that the method stalls.

\begin{figure}[t]
  \begin{tabular}{@{}ccc@{}}
   \includegraphics[width=.31\textwidth]{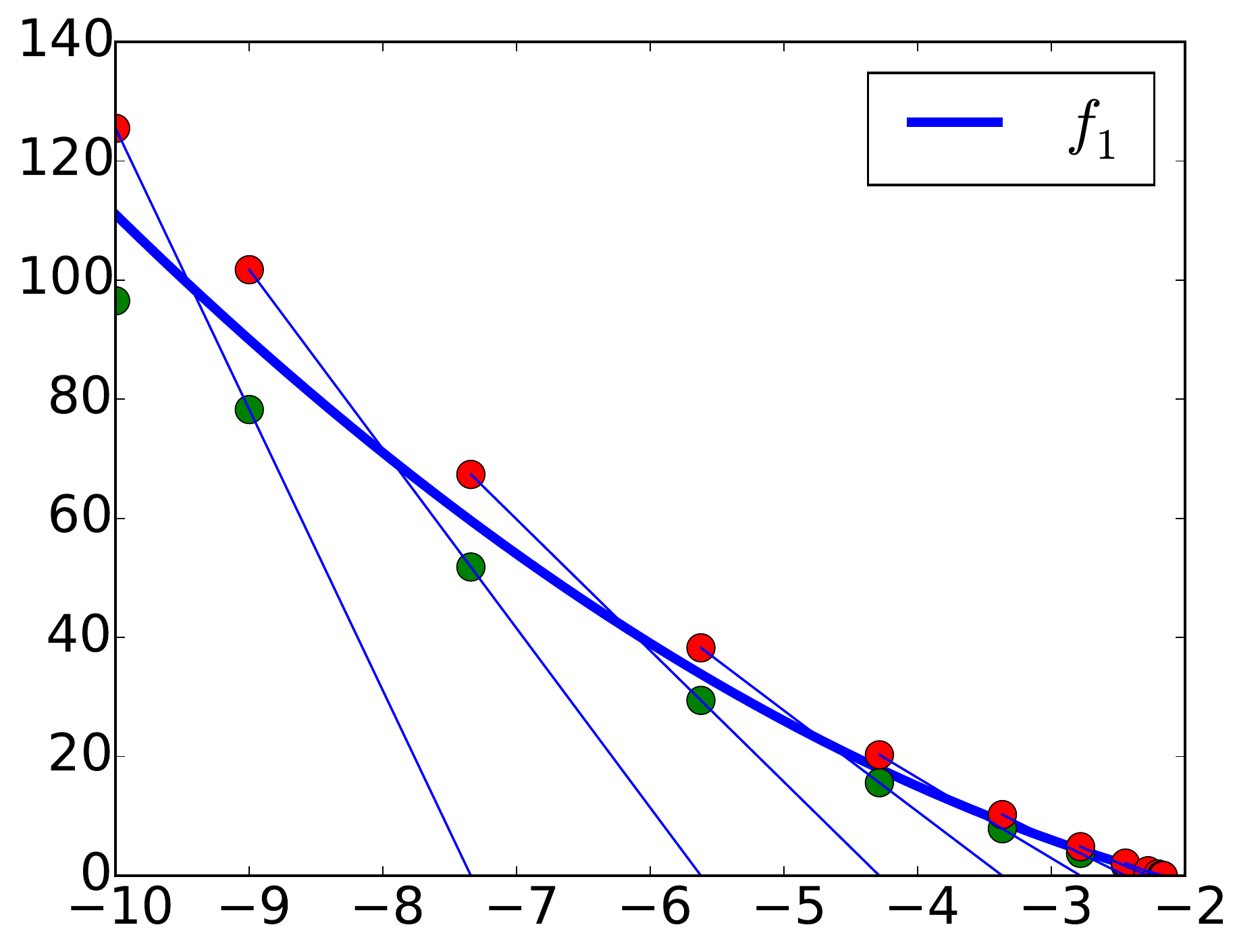}
  &\includegraphics[width=.31\textwidth]{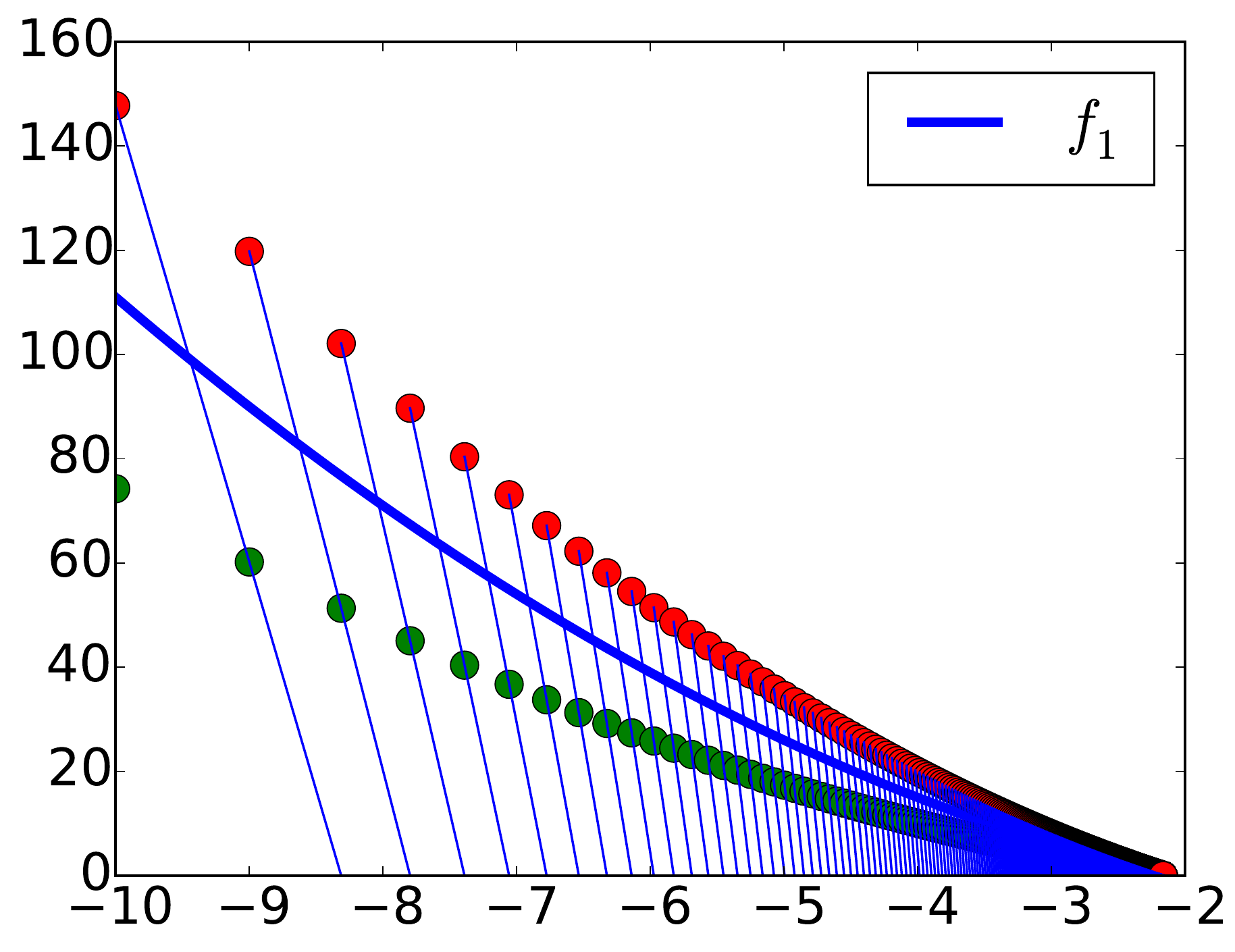}
  &\includegraphics[width=.31\textwidth]{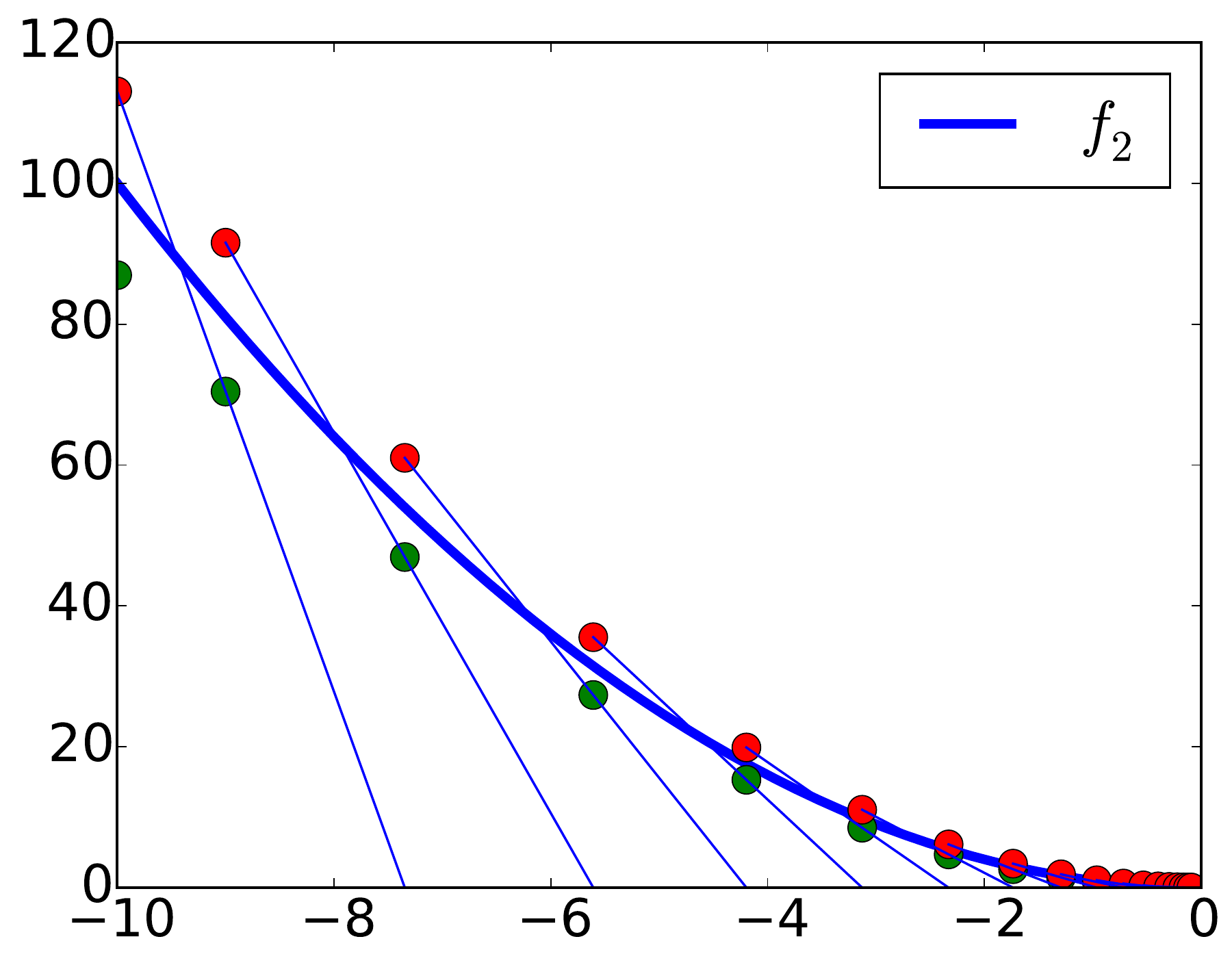}
 \\(a) $k=13,\ \alpha=1.3$ & (b) $k=770,\ \alpha=1.99$ & (c) $k=18,\ \alpha=1.3$
 \\\includegraphics[width=.31\textwidth]{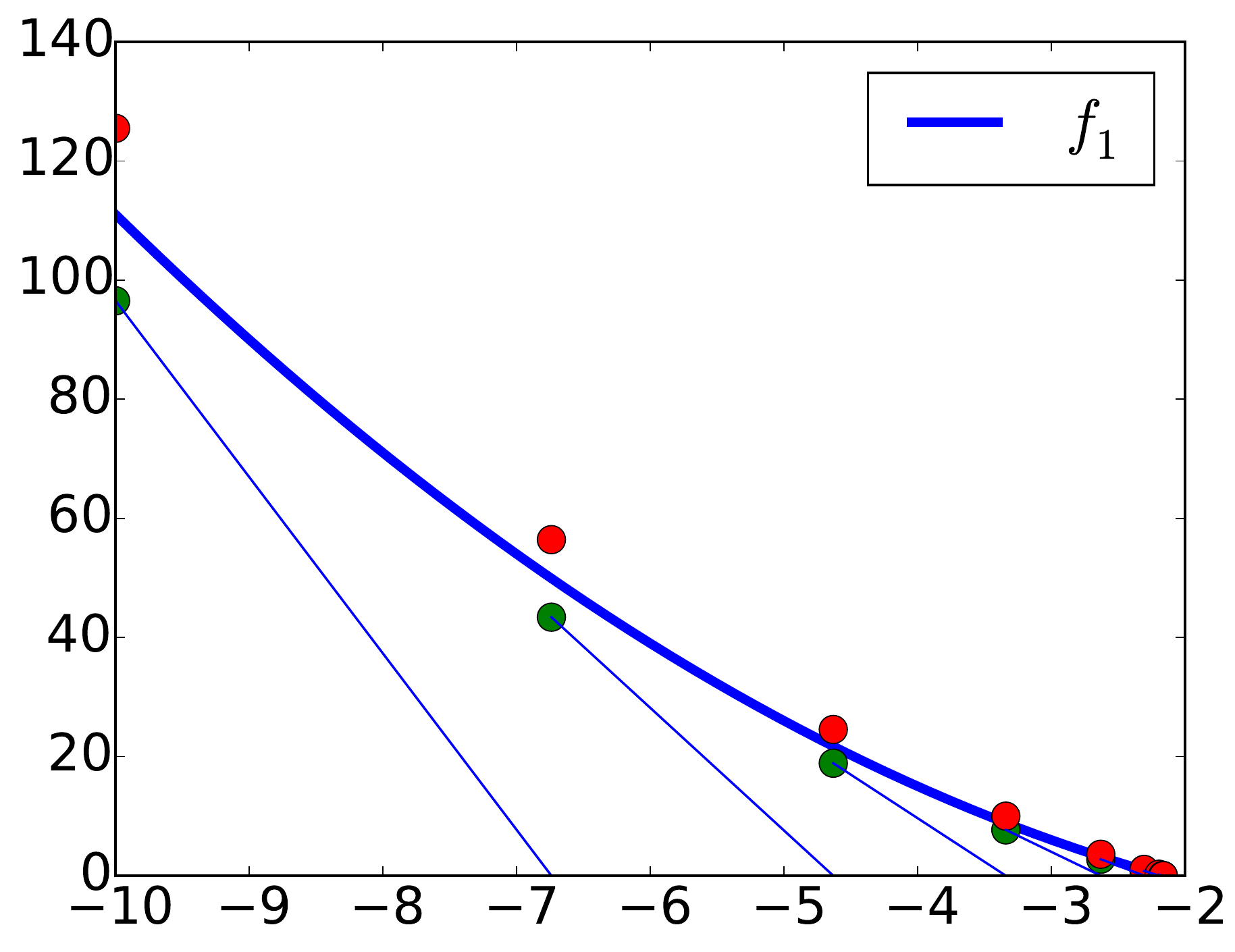}
  &\includegraphics[width=.31\textwidth]{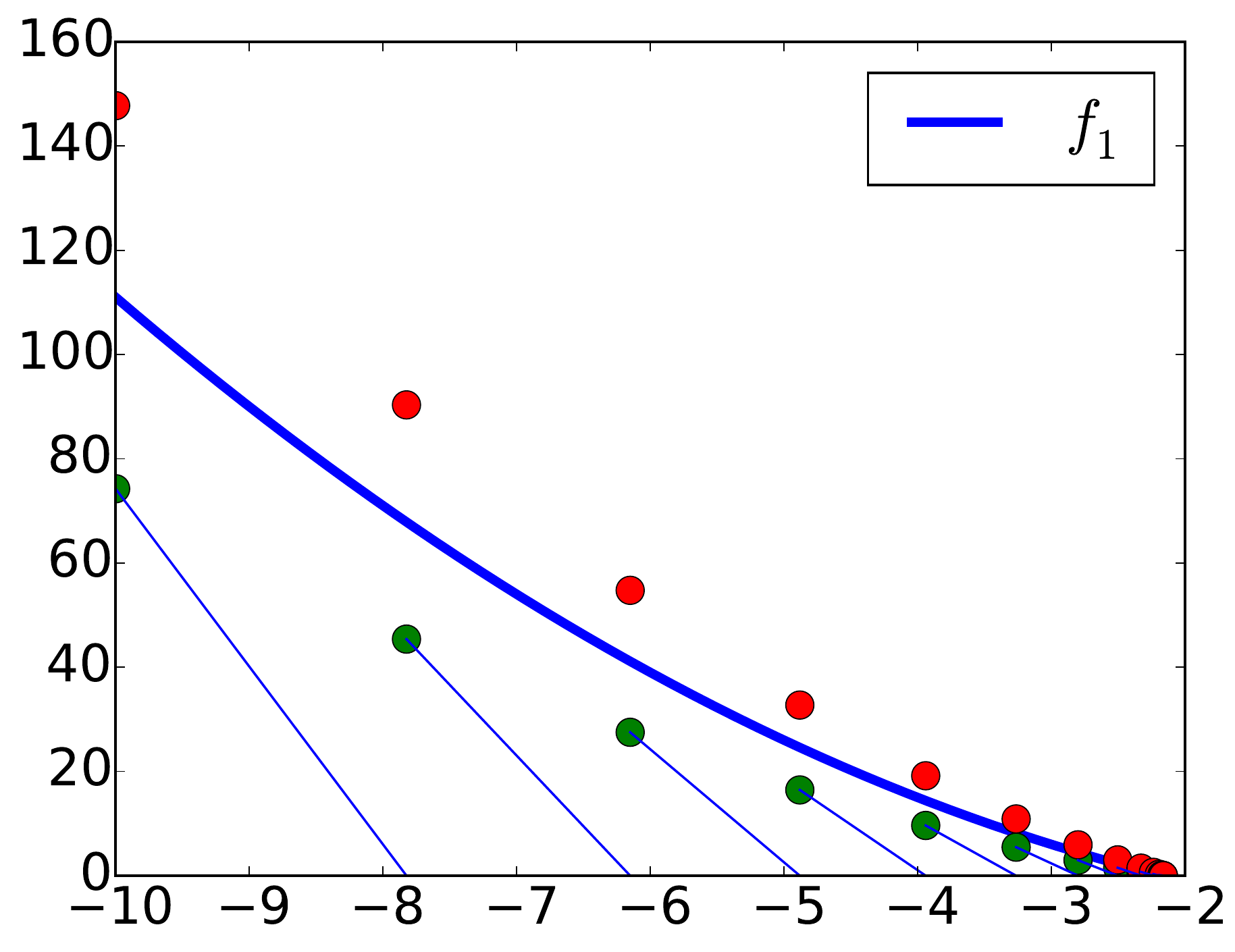}
  &\includegraphics[width=.31\textwidth]{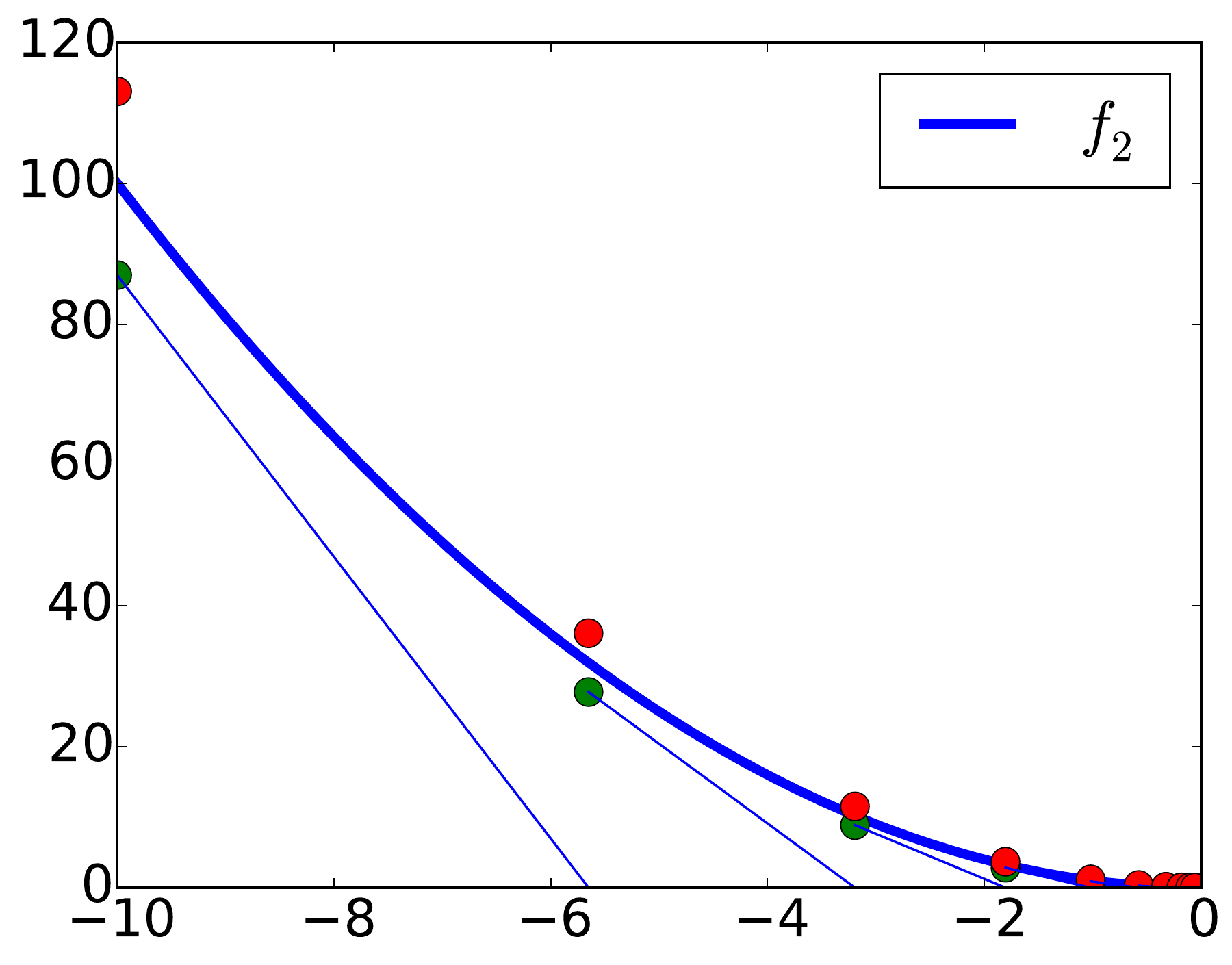}
 \\(d) $k=9,\ \alpha=1.3$ & (e) $k=15,\ \alpha=1.99$ & (f) $k=10,\ \alpha=1.3$
  \end{tabular}
  \caption{Inexact secant method (top row) and Newton method (bottom
    row) for root finding on the functions $f_1(\tau)=(\tau-1)^2-10$
    (first two columns) and $f_2(\tau)=\tau^2$ (last column).  Below
    each panel, $\alpha$ is the oracle accuracy, and $k$ is the number
    of iterations needed to converge, i.e., to reach
    $f_i(\tau_k)\le\epsilon$. For all problems,
    $\epsilon=10^{-2}$; the horizontal axis is $\tau$, and the vertical
    axis is $f_i(\tau)$.}
  \label{fig:sec_degen}
\end{figure}

\subsection{Inexact Newton}

The secant method can be improved by using approximate derivative
information (when available) to design a Newton-type method. We
design an inexact Newton method around an improved oracle that
provides global linear under-estimators of $f$. This approach has two main
advantages over the secant method. First, it is guaranteed to take
longer steps than the inexact secant method. Second, it locally
converges quadratically whenever $f$ is smooth, the values $f(\tau)$ 
are computed exactly, and the function has a nonzero (left) derivative at the
minimal root. To formalize these ideas, we use the following
strengthened version of an inexact evaluation oracle.

\begin{defn}[Affine minorant oracle]
  For a function $f:\bR_+ \to \bR$, an \emph{affine minorant oracle}
  is a mapping $\mathcal{O}_f$ that assigns to each pair
  $( t, \alpha ) \in [f > 0] \times [1,\infty)$ real numbers
  $(\ell,u,s)$ such that $0 < \ell \leq f(\tau)\leq u$ and
  $u/\ell \leq \alpha$, and the affine function
  $\tau' \mapsto \ell + s (\tau' - \tau)$ globally minorizes
  $f$.
\end{defn}

Algorithm~\ref{algo:newton} outlines a Newton method based on the
affine minorant oracle. The inexact Newton method enjoys global
convergence guarantees analogous to those of the inexact secant
method, as described by Theorem~\ref{thm:conv_inex_newt}; see Appendix~\ref{appendix:proofs} for the proof.

\begin{algorithm}[t]
  \DontPrintSemicolon 
  \KwData{Convex decreasing function
    $f\colon\bR_+\to\bR$ via an affine minorant oracle
    $\mathcal{O}_f$; target accuracy $\epsilon>0$; initial point
    $\tau_0$ with $f(\tau_0) > 0$; constant $\alpha\in (1,2)$.}
  $u_{-1}\gets +\infty$\;
  $k\gets0$\;
  \While{$u_k >\epsilon$}{
    $(\ell_k,u_k,s_k)\gets\cO_f(\tau_k,\alpha)$
      \Comment*{evaluate lower affine minorant oracle}
    $u_k\gets \min\{u_k,u_{k-1}\}$
      \Comment*{ensure upper bound decreases}
    $\tau_{k+1} \gets \tau_k - \ell_k/s_k$
      \Comment*{Newton iteration}
    $k\gets k+1$
  }
  \Return $\tau_{k}$\;
  \caption{Inexact Newton method}
  \label{algo:newton}
\end{algorithm}

\begin{theorem}[Linear convergence of the inexact Newton method]
  \label{thm:conv_inex_newt}
  The inexact Newton method (Algorithm ~\ref{algo:newton})
  terminates after at most
  \[ 
    k\leq \max \left\{ 1 + \log_{2/\alpha}\!
      \left( 2 C /\epsilon \right),\, 2 \right\}
  \]
  iterations, where $C := \max\{ |s_0|(\tau_*-\tau_0),\, \ell_0 \}$.
\end{theorem}

When we compare the two algorithms, it is easy to see that the Newton
steps are never shorter than the secant steps. Indeed, let
$(\ell_{k-1},u_{k-1},s_{k-1})=\cO_f(\tau_{k-1},\alpha)$ and
$(\ell_{k},u_{k},s_{k})=\cO_f(\tau_{k},\alpha)$ be the triples
returned by an affine minorant oracle at $\tau_{k-1}$ and $\tau_{k}$,
respectively. Then
\[ u_{k-1}\geq f(\tau_{k-1})\geq \ell_k + s_k (\tau_{k-1} - \tau_k),\]
which implies
\[
  s^{\hbox{\scriptsize secant}}_k
  := (u_{k-1}-\ell_k)/(\tau_{k-1}-\tau_k)\leq s_k
  =: s^{\hbox{\scriptsize newton}}_k.
\]
Therefore, the Newton step length
$-\ell_k/s^{\hbox{\scriptsize newton}}_k$ is at least as large as
the secant step length $-\ell_k/s^{\hbox{\scriptsize secant}}_k$.

As might be expected, the Newton method often outperforms the secant method
in practice. The bottom row of panels in
Figure~\ref{fig:sec_degen} shows the progress of the Newton method on
the same degenerate and nondegenerate test problems discussed
earlier. Note in particular that the Newton method performs relatively
well even when $\alpha$ is near its upper limit of 2; compare panels
(b) and (e) in the figure. In this set of experiments, we chose an
oracle with the same quality lower and upper bounds as the experiments
with secant, but has the least favorable (i.e., steepest) slope that
still results in a global minorant.

\subsection{Lower minorants from duality}\label{sec:lower minorants via duality}

Under what circumstances are affine minorant oracles of the value
function $v$ readily available? Not surprisingly, duality delivers an
answer. Suppose we can express the value function in dual form
$$
  v(\tau)=\max_{y}\ \Phi(y,\tau),
$$
where $\Phi$ is concave in $y$ and convex in $\tau$.  For example,
appealing to Fenchel duality, we may write
\begin{equation*}
  \begin{aligned}
    v(\tau)
    &= \min_{x\in \cX}\ \set{\rho(Ax-b)}{\varphi(x)\leq \tau}
  \\&= \min_{x\in \bR^n}\ \rho(Ax-b)
       + \delta_{\mathcal{X}\cap\,[\varphi\leq \tau]}(x)
  \\&= \max_{y\in\bR^m} \ip{y}{b} - \rho^{\star}(-y)
       - \delta^{\star}_{\mathcal{X}\cap\,[\varphi\leq \tau]}(A\T y),
	\end{aligned}
\end{equation*}
where the last equality holds provided that either the primal or the
dual problem has a strictly feasible point \cite[Theorem
3.3.5]{cov_lift}.  Hence, the Fenchel dual objective
\begin{equation}\label{eqn:dual-rep}
\Phi(y,\tau):=\ip b y -\rho^{\star}(-y) -\delta^{\star}_{\mathcal{X}\cap
  [\varphi\leq \tau]}(A\T y)
\end{equation}
yields an explicit representation for $\Phi$.  Note that convexity of
$\Phi$ in $\tau$ is immediate; see Lemma~\ref{lem:conc_supp}.

Many standard first-order methods that might be used as an oracle for
evaluating $v(\bar\tau)-\sigma$, generate both a lower bound
$\bar\ell$ and a dual certificate $\bar{y}$ that satisfy the equation
$\bar\ell=\Phi(\bar{y},\bar{\tau})-\sigma$. Examples include
saddle-prox \cite{saddle_prox}, Frank-Wolfe~\cite{jaggi:2013,FW-alg}, some projected (sub)gradient methods
\cite{dual_frank}, and accelerated versions
\cite{tseng:2010,tseng_f_order,Nesterov2005c}.  Whenever such a dual
certificate $\bar{y}$ is available, we have
\begin{equation} \label{eqn:lowe-min}
  \begin{aligned}
    v(\tau)-\sigma\geq \Phi(\bar{y},\tau)-\sigma
    &= \big(\Phi(\bar{y},\bar{\tau})-\sigma\big)
       + \big(\Phi(\bar{y},\tau)-\Phi(\bar{y},\bar{\tau})\big)
  \\&\geq \bar{\ell}+\bar{s}(\tau-\bar{\tau}),
\end{aligned}
\end{equation}
where $\bar{s}$ is any subgradient of $\Phi$ at $(\bar{y},\bar{\tau})$
with respect to $\tau$. Hence, an inexact evaluation oracle that uses
dual certificates can always be upgraded to an affine minorant oracle
provided that an element of the subdifferential
$\partial_{\tau} \Phi(y,\tau)$ can be evaluated. In the context
of~\eqref{eqn:dual-rep}, this amounts to being able to compute 
an element of 
$\partial_{\tau}\delta^{\star}_{\mathcal{X}\cap\, [\varphi\leq \tau]}(A\T y)$. 
Reassuringly, such subdifferential formulas
are readily available for a huge class of contemporary
problems~\cite[Equations 4.1b, 6.5d,
6.20]{AravkinBurkeFriedlander:2013}, and, in particular, for all the
problems discussed in the rest of the paper.

In some instances, lower-bounds on the optimal value of $\Qtau$
provided by an algorithm are seemingly not related to a dual
solution. A notable example of such a scheme is the Frank-Wolfe
algorithm, which has recently received much attention.  Supposing that
the function $\rho$ is smooth, the Frank-Wolfe method applied to the
problem $\Qtau$ iterates the following two steps:
\begin{equation}\label{eqn:FW}
  \left\{
    \begin{aligned}
      z_{k}&\ =\ \argmin_{z\in \mathcal{X}\cap [\varphi\leq \tau]}~ \langle A^T\nabla \rho(Ax_k-b),z\rangle\\
      x_{k+1} &\ =\ x_k+t_k(z_k-x_k)
    \end{aligned}
  \right.
\end{equation}
for an appropriately chosen sequence of step-sizes
$t_k$ (e.g.,
$t_k=\frac{2}{k+2}$). As the method progresses, it generates the upper
bounds
\[
  u_k=\min_{i=1,\ldots,k} \rho(Ax_i-b)
\]
on the optimal value of
$\Qtau$.  Moreover, it is easy to deduce from convexity that the
following are valid lower bounds:
\[
  \ell_k=\max_{i=1,\ldots,k} \left\{\rho(Ax_i-b)+ \langle  A^T
    \nabla\rho(Ax_i-b) , z_i-x_i\rangle\right\}.
\]
Jaggi~\cite{jaggi:2013} provides an extensive discussion.  If the step
sizes $t_k$ are chosen appropriately, the gap satisfies
$u_k-\ell_k\leq \mathcal{O}(D^2L/k)$, where the diameter $D$ of the
feasible region and the Lipschitz constant $L$ of the gradient of the
objective function of \Qtau\ are measured in an arbitrary norm.
Harchaoui, Juditsky, and Nemirovski~\cite{cond_grad} observe how to
deduce from such lower bounds $\ell_k$ an affine minorant of the value
function $v$, leading to a level-set scheme based on Newton's method.

On the other hand, one can also show that the lower bounds $\ell_k$ are
indeed generated by an explicit candidate dual solution, and hence the
Frank-Wolfe algorithm (and its variants) fit perfectly in the above
framework based on dual certificates.  To see this, consider the
Fenchel dual
\[
  \maximize{y\in\bR^m}\quad \Phi(y,\tau)=\ip{y}{b} - \rho^{\star}(-y)
       - \delta^{\star}_{\mathcal{X}\cap
  [\varphi\leq \tau]}(A\T y)
\]
of \Qtau. Then for the candidate dual solutions
$y_i:=-\nabla \rho(Ax_i-b)$, we successively deduce
\begin{align*}
       \Phi(y_i,\tau)&=\langle y_i,b\rangle-\rho^{\star}(-y_i)-\langle A^Ty_i,z_i\rangle&\qquad \mbox{[definition of $z_i$]}\\
       &=\langle y_i,b\rangle+\Big(\rho(Ax_i-b)+\langle y_i, Ax_i-b\rangle\Big) -\langle A^Ty_i,z_i\rangle &\qquad\mbox{[Fenchel-Young inequality]}\\
       &=\rho(Ax_i-b)+\langle A^T\nabla \rho(Ax_i-b),z_i-x_i\rangle.
\end{align*}
Thus, the lower bounds $\ell_k$ are simply equal to
$\ell_k=\max_{i=1,\ldots,k} \Phi(y_i,\tau)$, and affine minorants on
the value function $v$ are readily computed from the dual iterates
$y_k$ and the derivatives
$\partial_{\tau}\delta^{\star}_{\mathcal{X}\cap\, [\varphi\leq
  \tau]}(A\T y_k)$.

\section{Refinements}\label{sec:refinement}
This section can be considered as an aside in our main
exposition. Here, we address two questions that arise in the
application of our root-finding approach: how best to apply the
algorithm to problems with linear least-squares constraints, and how
to recover a feasible point.

\subsection{Least-squares misfit and degeneracy}
\label{sec:ls-misfit}

Particularly important instances of problem \Psig\ arise when the
misfit between $Ax$ and $b$ is measured by the 2-norm, i.e.,
$\rho=\|\cdot\|_2$. In this case, the objective of the level-set
problem \Qtau\ is $\|Ax-b\|_2$, which is not differentiable whenever
$Ax=b$. Rather than applying a nonsmooth optimization scheme, an
apparently easy fix is to replace the constraint in \Psig\ with its
equivalent formulation $\tfrac12\norm{Ax-b}_2^2\le\tfrac12\sigma^2$,
 leading to the pair of problems
\begin{alignat}{4}
  &\minimize{x\in\cX} &\quad& \varphi(x)
                      &\quad& \st
                      &\quad& \tfrac12\|Ax-b\|_2^2\le\tfrac12\sigma^2,
    \tag{$\cP^2_\sigma$}
\\&\minimize{x\in\cX} && \tfrac12\|Ax-b\|^2
                      &&\st
                      && \varphi(x)\le\tau.
  \tag{$\cQ^2_\tau$}
\end{alignat}
Throughout this section, the problems \Psig\ and \Qtau\ continue to
define the original formulations without the squares.

This straightforward adaptation, however, presents some numerical
difficulties. Following the strategy outlined in the previous
sections, the root finding procedure for $\Psig^2$ would be
automatically applied to the function
\[
  f_2(\tau):=\half v^2(\tau)-\half\sigma^2,
\]
where $v$ is the value function corresponding to the original
(unsquared) level-set problem \Qtau. Clearly, the function $f_2$ is
degenerate at each of its roots. As a result, the secant and Newton
root-finding methods, respectively, would not converge locally
superlinearly or quadratically---even if the values $v(\tau)$ are
evaluated exactly. Moreover, we have observed empirically that this
issue can in some cases cause numerical schemes to stagnate.

A simple alternative avoids this pitfall: apply the root-finding
procedure to the function \[f_1(\tau) := v(\tau)-\sigma\]
corresponding to the value function of \Qtau, but solve $\Qtau^2$ to
approximately evaluate $f_2$ and consequently to approximately
evaluate $f_1$. The oracle definitions required for the secant
(Algorithm~\ref{algo:secant}) and Newton (Algorithm~\ref{algo:newton})
methods require suitable modification. For secant, the modifications
are straightforward, but for Newton, care is needed in order to obtain
the correct affine minorants of $f_1$ from those of $f_2$. The
required modifications are described in turn below.

\subsubsection*{Secant}

For the secant method applied to the function $f_1$, we derive an inexact
evaluation oracle from an inexact evaluation oracle for $f_2$ as
follows. Suppose that we have approximately solved $\Qtau^2$ by an inexact-evaluation oracle
\begin{equation}
  \label{eq:oracle-f2}
  \cO_{f_2}\big(\tau, \alpha^2\big)
  = \left(\half\ell^2-\half\sigma^2,\ \half u^2-\half\sigma^2\right),
\end{equation}
where we have specified the relative accuracy between the lower and
upper bounds to be $\alpha^2$. Assume, without loss of generality,
that $u,\ell\geq 0$. Then clearly $u$ and $\ell$ are upper and lower
bounds on $v(\tau)$, respectively.  It is now straightforward to
deduce 
\begin{equation} \label{eq:sec}
  0\leq \ell-\sigma\leq f_1(\tau) \leq u-\sigma
  \qquad \hbox{and}\qquad 
  \frac{ u-\sigma}{\ell-\sigma}
  \leq \sqrt{\frac{ u^2-\sigma^2}{ \ell^2-\sigma^2}}\leq \alpha.
\end{equation}
Hence an inexact function evaluation oracle for $f_2$ 
yields an inexact evaluation oracle for $f_1$.

\subsubsection*{Newton}

Newton's method in this setting is slightly more intricate: the nuance
is in obtaining a valid affine minorant of $f_1$.
We use the respective objectives of the dual problems corresponding to
$\Qtau$ and $\Qtau^2$, given by
\begin{align*}
  \Phi_1(y,\tau) &:=
  \ip b y
  - \delta^{\star}_{\cX\cap[\varphi\leq \tau]}(A^T y)
    - \delta_{\ball_2}(y),
\\
  \Phi_2(y,\tau) &:=
  \ip{b}{y}
  -\delta^{\star}_{\cX\cap[\varphi\leq \tau]}(A^T y)
   -\half\|y\|^2_2.
\end{align*}
As described by~\eqref{eq:oracle-f2}, an inexact solution of $\Qtau^2$
delivers values $\ell$ and $u$ that satisfy~\eqref{eq:sec}. Suppose
that the oracle additionally delivers a dual certificate $y$ that
satisfies $\Phi_2(y,\tau)=\half \ell^2$.  Let
$s\in \partial_{\tau} \Phi_2(y,\tau)$ be any subgradient. The
following result establishes that
\[
 (\hat\ell,\, u,\, s/\|y\|_2)
 \textt{with}
 \hat{\ell}:=\Phi_1\left(y/\|y\|_2,\,\tau\right),
\]
defines a valid affine minorant for $f_1$.

\begin{proposition}\label{cl:ineq_quad}
  The inequalities
  $$
  0\leq \hat\ell-\sigma\leq f_1(\tau) \leq u-\sigma
  \qquad\hbox{and}\qquad
  (u-\sigma)/(\hat\ell-\sigma) \leq \alpha
  $$ 
  hold, and the linear functional
  $\tau' \mapsto (\hat\ell-\sigma)-(s/\|y\|_2)(\tau'-\tau)$
  minorizes $f_1$.
\end{proposition}
The proof is given in Appendix~\ref{appendix:proofs}. In summary, if
we wish to obtain a super-optimal and $\epsilon$-feasible solution to
$\Psig$, in each iteration of the Newton method we must evaluate
$f_2(\tau)$ up to an absolute error of at most
$\frac{1}{2}(1-1/\alpha)^2\epsilon^2$. Indeed, suppose that in the
process of evaluation, the oracle $\cO_{f_2}\big(\tau, \alpha^2\big)$
achieves $u$ and $l$ satisfying
\[
  \half u^2-\half \ell^2\leq
  \half (1-1/\alpha)^2\epsilon^2.
\]
Then we obtain the inequality
\[
  u-\ell= \sqrt{(u-\ell)^2}
  \leq \sqrt{u^2-\ell^{2}} \leq (1-1/\alpha)\epsilon,
\]
Thus, by the discussion following Definition~\ref{defn:oracle}, either
the whole Newton scheme can now terminate with
$f_1(\tau)\leq \epsilon$ or we have achieved the relative accuracy
$(u-\sigma)/(\ell-\sigma)\leq \alpha$ for the oracle.

\subsection{Recovering feasibility}\label{sec:rec_feas}

A potential shortcoming of the level-set approach is that the computed
solutions are only $\epsilon$-feasible. Some applications may demand
feasible solutions. A straightforward remedy is to project the
computed $\epsilon$-feasible point onto the original constraint set
\( \set{x\in\cX}{\rho(Ax-b)\le\sigma}.  \)
However, this operation can be computationally impractical; for
example, access to the matrix $A$ is often only available through
matrix vector products. An alternative is provided by
Renegar~\cite{ren_conic}, who suggests an inexpensive
radial-projection scheme for conic optimization that generates a
feasible point while still preserving some notion of optimality. The
approach requires knowledge of a point $e$ strictly feasible for the
original problem, and obtains a feasible point $x$ whose optimality is
measured with respect to $e$, i.e.,
\[
  \frac{\varphi(x)-\OPT}{\varphi(e)-\OPT}\le\delta
\]
for some small positive parameter $\delta$.

To explain the approach, fix some target $\delta< 1$ and suppose that $e\in\cX$ is strictly feasible
for \Psig, i.e.,
\[
\rho(Ae-b)<\sigma.
\]
Suppose also that a point $z\in \mathcal{X}$ is super-optimal and $\epsilon$-feasible
for \Psig:
\[\varphi(z)\leq \textrm{OPT} \quad \textrm{ and }\quad
  \sigma < \rho(Az-b)\le\sigma+\epsilon,
  \textt{with}
  \epsilon := \delta\big[\sigma-\rho(Ae-b)\big].
\]
 These relationships imply the inequality
\[
  \alpha :=
  \frac{\rho(Az-b) - \sigma}{\rho(Az-b) - \rho(Ae-b)}
  \leq \delta.
\]
Set $x := z + \alpha ( e - z)$, which is the radial projection of $z$
towards the feasible point $e$.  It follows from convexity that
$\varphi(x) \leq (1-\alpha) \varphi(z) + \alpha \varphi(e)$. Subtract
$\OPT$ from both sides and rearrange terms to obtain
\[
 \frac{ \varphi(x) - \OPT }{ \varphi(e) - \OPT} \leq (1-\alpha) \frac{
  \varphi(z) - \OPT }{\varphi(e) - \OPT} + \alpha \le \alpha\le \delta.
\]
It only remains to show that the radial projection $x$ is
feasible. The inclusion $x\in \mathcal{X}$ follows from convexity of
$\cX$. Use the definition
of $\alpha$, together with the convexity of $\rho$, to obtain
\[
  \rho(Ax-b)\le\rho(Az-b)-\alpha[\rho(Az-b)-\rho(Ae-b)] = \sigma,
\]
which establishes feasibility of $x$.

\section{Some problem classes}
\label{sec:pr_class}

There is a surprising variety of useful problems that can be treated
by the root-finding approach. These include problems from sparse
optimization, with applications in compressed sensing and sparse
recovery, generalized linear models, which feature prominently in
statistical applications, and conic optimization, which includes
semidefinite programming. The following sections are in some sense a
``cookbook'' that describes how features of particular problems can be
combined to apply the root-finding approach. In some cases, such as
with conic optimization, we have the opportunity to derive unexpected
algorithms.

\subsection{Conic optimization}

The general conic problem (CP) has the form
\begin{equation}
  \label{eq:conic-problem}
    \minimize{x} \quad \ip c x \quad\st\quad \cA x=b,\ x\in\cK,
    \tag{CP}
\end{equation}
where $\cA: E_1 \to E_2$ is a linear map between Euclidean spaces, and
$\cK \subset E_1$ is a proper, closed, convex cone. The familiar forms
of this problem include linear programming (LP), second-order cone
programming (SOCP), and semidefinite programming (SDP). Ben-Tal and
Nemirovski~\cite{mode_lec} survey an enormous number of applications
and formulations captured by conic programming.

There are at least two possible approaches for applying the level-set framework.
The first exchanges the roles of the original objective $\ip c x$ with
the linear constraint $Ax=b$, and brings a least-squares term into the
objective; the second approach moves the cone constraint $x\in\cK$
into the objective via a kind of distance function. This yields two
distinct algorithms for the conic problem. The two approaches are
summarized in Table~\ref{tab:cp_flip}. Note that it is possible
to consider conic problems with the more general constraint
$\rho(Ax-b)\le\sigma$, but here we restrict our attention to the
simpler affine constraint, which conforms to the standard form of
conic optimization.

\begin{table}[t]
\centering\small
\begin{tabular}{p{.75in}lll}
  \toprule
  Problem
& \multicolumn{1}{c}{\Psig}
& \multicolumn{1}{c}{\Qtau}
& \multicolumn{1}{c}{Dual of \Qtau}
\\\midrule
 \vtop{CP\par least-squares\par level} &
    $\begin{array}[t]{cl}
       \displaystyle\min_{x} & \ip c x
       \\\hbox{s.t.} & \begin{aligned}[t]
           \cA x&=b \\ x&\in\cK
       \end{aligned}
     \end{array}$
    & $\begin{array}[t]{cl}
       \displaystyle\min_{x} & \|\cA x-b\|_{2}
         \\\hbox{s.t.} & \begin{aligned}[t]
            \ip c x&\le\tau\\x&\in \cK
            \end{aligned}
          \end{array}$
    & $\begin{array}[t]{cl}
         \displaystyle\max_{y,\ \mu\ge0} & \ip{b}{y} - \mu\tau
         \\\hbox{s.t.} & \begin{aligned}[t]
           \|y\|_{2}&\le1\\\mu c - \cA^{*}y&\in \mathcal{K}^*
         \end{aligned}
       \end{array}$
\\ \midrule
 \vtop{CP\par cone\par level} &
   $\begin{array}[t]{cl}
              \displaystyle\min_{x} & \ip c x
            \\\textrm{s.t.} & \begin{aligned}[t]\cA x&=b\\x&\in \mathcal{K}\end{aligned}
              \end{array}$
         & $\begin{array}[t]{cl}
              \displaystyle\min_{x} & -\lambda_{\min}(x)
            \\\textrm{s.t.} & \begin{aligned}[t]\cA x&=b\\\ip c  x&\le\tau\end{aligned}
              \end{array}$
         & $\begin{array}[t]{cl}
              \displaystyle\max_{y,\ \mu\ge0} & \ip{b}{y}-\mu\tau
            \\\textrm{s.t.} & \begin{aligned}[t]\ip {\mu c - \mathcal{A}^* y}{e} &=1\\ \mu c - \cA^{*}y&\in \mathcal{K}^*\end{aligned}
              \end{array}$
\\ \bottomrule
\end{tabular}
\caption{Least-squares and conic level-set problems for conic optimization. In these
  examples, we require $\cA x=b$.}
\label{tab:cp_flip}
\end{table}

\subsubsection{First approach: least-squares level set}
\label{sec:least-squares-level}

To get started with this approach, we make the blanket assumption that
we know a {\em strictly feasible} vector $\yhat$ for the dual
of~\eqref{eq:conic-problem}:
\begin{equation*}
\label{eq:conic-problem_dual}
\maximize{y} \quad \ip b y \quad\st\quad c-\mathcal{A}^*y\in\cK^*.
\end{equation*}
Thus $\yhat$ satisfies 
$\chat:=c - \cA^* \yhat \in \interior \mathcal{K}^*$. %
A simple calculation shows that minimizing
the new objective $\ip{\chat}{x}$ only
changes the objective of CP by a constant: for all $x$ feasible for
CP, we now have
\[
\ip{\chat}{x}=\ip{c}{x}-\ip{\cA x}{\yhat}=\ip{c}{x}-\ip{b}{\yhat}.
\]
In particular, we may assume $b \neq 0$, since otherwise, the
origin is the trivial solution for the shifted problem.  Note that in
the important case $c\in \interior \mathcal{K}$, we can simply set
$\yhat=0$, which yields the equality $c=\chat$.

We now illustrate the computational complexity of applying the
root-finding approach to solve (CP) using the level-set problem
\begin{equation}
  \label{eq:lp-qtau}
  \minimize{x} \quad \|\cA x - b\|_2
    \quad\st\quad \ip{\chat}{x} \le \tau,\ x\in\cK.
\end{equation}
Our aim is
then to find a root of \eqref{eqn:root}, where $v$ is the value
function of~\eqref{eq:lp-qtau}.  The top row of
Table~\ref{tab:cp_flip}, gives
the corresponding dual
\[
  \maximize{y,~\,\mu\geq 0} \quad \ip b y - \mu \tau
  \quad\st\quad \|y\|_2\le 1,\ \mu c - \cA^* y \in \mathcal{K}^*
\]
of the level-set problem.  We use $\tau_0=0$ as the initial
root-finding iterate. Because of the inclusion
$\chat\,\in\interior \mathcal{K}^*$, we deduce that $x=0$ is the only
feasible solution to \eqref{eq:lp-qtau}, which yields $v(0) = \|b\|_2$
and the exact lower bound $\ell_0=\|b\|_2$.  The corresponding dual
certificate is $(\bar{y},\bar{\mu}) = (b/\|b\|_2, \bar \mu)$, where
\begin{equation}\label{eqn:mu}
  \bar{\mu}:=\min_\mu\left\{\mu \chat-\frac{\mathcal{A}^*b}{\|b\|_2}\in \cK^*\right\}.
\end{equation}
Note the inequality $\bar{\mu}>0$, because otherwise we would deduce
$\mathcal{A}^*b\in -\mathcal{K}^*$, implying the inequality
$\|b\|^2_2=\langle b,\mathcal{A}x\rangle=\langle\mathcal{A}^*b,
x\rangle\leq 0$
for any feasible $x$. This contradicts our assumption that $b$ is
nonzero. In the case where $\cK$ is the nonnegative orthant and
$\chat=e$, the number $\bar \mu$ is simply the maximal coordinate of
$\mathcal{A}^*b/\|b\|_2$; if $\cK$ is the semidefinite cone and
$\chat=I$, the number $\bar \mu$ is the right-most eigenvalue of
$\mathcal{A}^*b/\|b\|_2$.
With these
values, Theorem~\ref{thm:conv_inex_newt} asserts that within
$\cO\big(\log_{2/\alpha} 2C/\epsilon\big)$ inexact Newton iterations,
where $\alpha$ is the accuracy of each subproblem solve and
\[
  C =
  \max
  \left\{
        \bar{\mu}\cdot\left(\text{OPT}- \langle b,\yhat\rangle\right)
        ,\|b\|_2
      \right\},
\]
the point $x\in \mathcal{K}$ that yields the final upper bound in
\eqref{eq:lp-qtau} is a super-optimal and $\epsilon$-feasible solution
of the shifted CP, i.e.,
$$\langle \chat,x\rangle\leq \OPT-\langle \yhat,b\rangle\qquad \textrm{ and } \qquad \|\mathcal{A}x-b\|_2\leq \epsilon.$$
To see how good the obtained point $x$ is for the original CP (without
the shift), note that
\[
  \langle \chat,x \rangle =\langle c,x\rangle-\langle
  \mathcal{A}^*\yhat,x\rangle=\langle c,x\rangle-\langle
  \yhat,\mathcal{A}x-b\rangle-\langle \yhat,b\rangle\geq \langle
  c,x\rangle-\langle \yhat,b\rangle-\epsilon\|\yhat\|_2,
\]
and hence $\langle c,x\rangle \leq \OPT+\epsilon \|\yhat\|_2$. In particular, in the important case where $c\in \interior \cK^*$, we deduce super-optimality $\langle c,x\rangle \leq \OPT$ for the target problem CP.

Each Newton root-finding iteration requires an approximate solution
of~\eqref{eq:lp-qtau}. As described in \S\ref{sec:ls-misfit}, we
obtain this approximation by instead solving its smooth formulation
with the squared objective $\half\|\cA x-b\|_2^2$.  Let $L:=\|\cA\|_2^2$ be the
Lipschitz constant for the gradient $\cA^T(\cA x-b)$, and let $D$ be the
diameter of the region $\{x\mid\ip \chat x = 1,\
x\in \cK\}$, which is finite by the inclusion $\chat\in \interior \cK^*$. %
Thus, in order to evaluate
$v$ to an accuracy $\epsilon$, we may apply an accelerated
projected-gradient method on the squared version of the problem to an
additive error of $\frac{1}{2}(1-1/\alpha)^2\epsilon^2$ (see end of \S\ref{sec:ls-misfit}), which
terminates in at most
\[
  \cO \left(
        \frac{\sqrt{L} \cdot\tau D}{\epsilon(1-1/\alpha)}
      \right)
 = \cO \left(
        \frac{\|A\|_2\cdot D\cdot \left(\hbox{OPT}-\langle b,\yhat \rangle\right)  }{\epsilon(1-1/\alpha)}
      \right)
\]
iterations~\cite[\S6.2]{bertsekas:2015}. Here, we have used the
monotonicity of the root finding scheme to conclude 
$\tau\leq \hbox{OPT}-\langle b,\yhat \rangle$. When $\mathcal{K}$ is
the non-negative orthant, each projection can be accomplished with
$\cO(n)$ floating point operations~\cite{BRUCKER1984163}, while for
the semidefinite cone each projection requires an eigenvalue
decomposition. More generally, such projections can be quickly found
as long as projections onto the cone $\mathcal{K}$ are available; see
Remark~\ref{rem:proj_base}. We note that an improved complexity bound
can be obtained for the oracles in the LP and SDP cases by replacing
the Euclidean projection step with a Bregman projection derived from
the entropy function; see e.g., Beck and
Teboulle~\cite{beck2003mirror} or Tseng~\cite[\S3.1]{tseng:2010}. We
leave the details to the reader.

\begin{figure}[t]
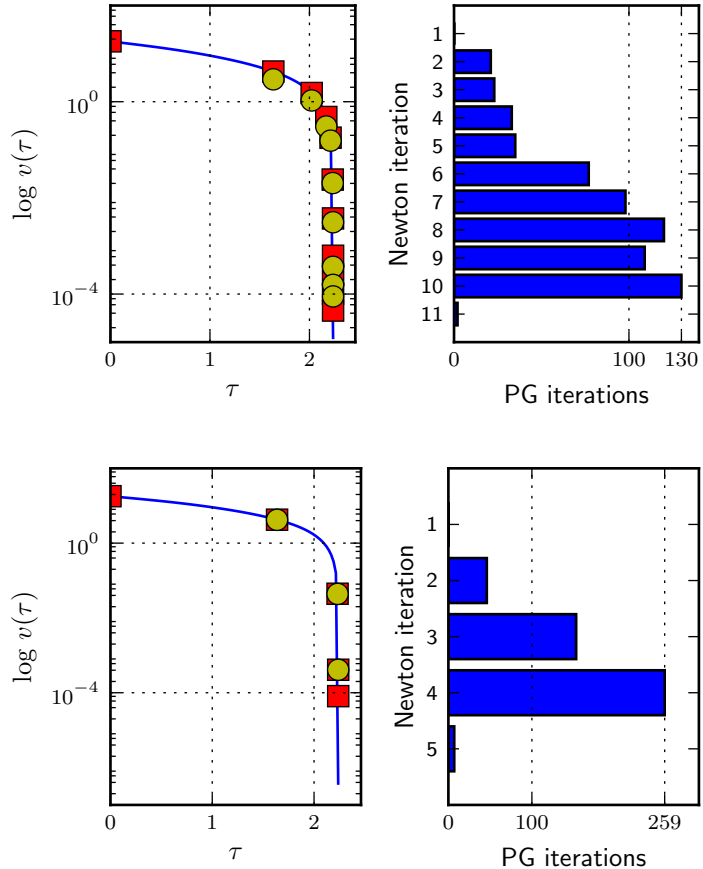

  \centering
  \input{figs/nnlp_1_8}
\\\input{figs/nnlp_1_01}
\caption{Progress of the root-finding method for a linear program. The
  panels on the left depict the graph of $v(\tau)$ (solid line),
  and the squares and circles, respectively, show the upper and
  lower bounds computed using an optimal projected-gradient
  method. The horizontal log scale results in a value function that
  appears nonconvex. The panels on the right show the number of
  projected-gradient iterations for each Newton step. Top panels:
  $\alpha=1.8$. Bottom panels: $\alpha=1.01$.}
\end{figure}

In summary, we can obtain a point $x\in \mathcal{K}$ that satisfies
$$\langle c,x\rangle\leq \OPT+\epsilon \|\yhat\|_2 \qquad \textrm{ and }\qquad \|\cA x-b\|_2\leq \epsilon $$
in at most
\[
  \cO \left(
  \frac{\|A\|_2\cdot D\cdot \left(\hbox{OPT}-\langle b,\yhat \rangle\right)  }{\epsilon(1-1/\alpha)}
  \right)
  \cdot
  \cO
  \left( \log_{2/\alpha} \frac{\max
  	\left\{
  	\bar{\mu}\cdot\left(\text{OPT}- \langle b,\yhat\rangle\right)
  	,\|b\|_2
  	\right\}}{\epsilon }
  \right)
\]
iterations of an accelerated projected-gradient method, where $\bar{\mu}$ is defined in \eqref{eqn:mu}.  Reassuringly,
the complexity bound depends on all the expected quantities.%

\subsubsection{Second approach: conic level set}

Renegar's recent work~\cite{ren_conic} on conic optimization inspires
a possible second level-set approach based on interchanging the roles
of the affine objective and the conic constraint
in~\eqref{eq:conic-problem}. A key step is to define a convex function
$\kappa$ that is nonnegative on the cone $\cK$, and positive
elsewhere, so that it acts as a surrogate for the conic constraint,
i.e.,
\begin{equation} \label{eq:8}
  \kappa(x)\le0 \textt{if and only if} x\in\cK.
\end{equation}
The conic optimization problem then can be expressed equivalently in
entirely functional form as
\begin{equation} \label{eq:2}
  \minimize{x} \quad\ip c x \quad\st\quad\cA x=b,\ \kappa(x)\le0,
\end{equation}
which allows us to define the level-set problem
\begin{equation} \label{eq:9}
  \minimize{x}\quad\kappa(x)\quad\st\quad\cA x=b,\ \ip c x \le\tau.
\end{equation}

Renegar gives a procedure for constructing a suitable surrogate
function $\kappa$ under the assumption that $\cK$ has a nonempty
interior: choose a point $e\in\interior\cK$ and define
$\kappa(x)=-\lambda\submin(x)$, where
\[
  \lambda\submin(x)
  := \inf\, \{ \lambda \mid x - \lambda e \not \in \cK \}.
\]
In the case of the PSD cone, we may take $e = I$, and then
$\lambda\submin$ yields the minimum eigenvalue function, which
explains the notation. As is shown in
\cite[Prop.~2.1]{ren_conic}, the function $\lambda_{\min}$ is
Lipschitz continuous (with modulus one) and concave, as would be
necessary to apply a subgradient method for minimizing $\kappa$. Renegar derives a novel algorithm along with complexity bounds for CP using the $\lambda_{\min}$ function. 
 A rigorous methodology for applying the level-set scheme, as described in the current paper, requires further research. It is an intriguing research agenda to unify Renegar's explicit complexity bounds with the proposed level-set approach.
We note in passing 
that the dual of the resulting level-set problem, needed to apply the lower
affine-minorant root-finding method, is shown in the second row of
Table~\ref{tab:cp_flip}, and can be derived using the conjugate of
$\lambda\submin$; see Lemma~\ref{rem:conjugate_lamdbdaMin}.

In principle, the main requirement of our level-set approach
is that the surrogate function that satisfies~\eqref{eq:8} yields the
equivalent formulation~\eqref{eq:2}.  Depending on the algorithms
available for solving the level-set problem~\eqref{eq:9}, it may be
convenient to define a function $\kappa$ with certain useful
properties. For example, we might choose to define the differentiable
surrogate function
\[
  \kappa = \half\di_\cK^2,
  \textt{where}
  \di_\cK(x):=\inf_{z\in\cK}\ \norm{x-z}
\]
measures the distance to the cone $\cK$.

Note the significant differences between the least-squares and conic
level-set problems~\eqref{eq:lp-qtau} and \eqref{eq:9}. For the sake
of discussion, suppose that $\cK$ is the positive semidefinite cone.
The least-squares level-set problem has a smooth objective whose
gradient can be easily computed by applying the operator $\cA$ and its
adjoint, but the constraint set still contains the explicit
cone. Projected-gradient methods, for example, require a full
eigenvalue decomposition of the steepest-descent step, while the
Frank-Wolfe method requires only a single rightmost eigenpair
computation.  The latter level-set problem, however, can require a
potentially more complex procedure to compute a gradient or
subgradient, but has an entirely linear constraint set. In this case,
projected (sub)gradient methods require a least-squares solve for the
projection step.

\subsection{Gauge optimization}
\label{sec:nonsmooth_reg}
In this section, we illustrate the general applicability of the
level-set approach to regularized data-fitting problems by restricting
the convex functions $\varphi$ and $\rho$ to be {\em gauges}---i.e.,
functions that are additionally nonnegative, positively homogeneous,
and vanish at the origin. Throughout, we assume that the side
constraint $x\in \mathcal{X}$ is absent from the formulation \Psig.  A
large class of problems of this type occurs in sparsity
optimization. Basis pursuit (and its ``denoising'' variant
BP$_\sigma$)~\cite{CheDonSau:99} was our very first example in
\S\ref{sec:intro}, and many related problems can be similarly
expressed.
The first two columns
of Table~\ref{tab:ns_reg} describe various formulations of current
interest, including basis pursuit denoising (BPDN), low-rank matrix
recovery~\cite{Fazel:2002,CandTao:2009}, a sharp version of the
elastic-net problem~\cite{EN_2005}, and gauge optimization \cite{FriedlanderMacedoPong:2014}
in its standard form. %
The third column shows the level-set problem
\ref{eqn:sp} needed to evaluate the value function $v(\tau)$, while the fourth column shows the slopes needed to implement the Newton scheme.

The dual representation~\eqref{eqn:dual-rep} can be specialized for
this family, and requires some basic facts regarding a gauge function
$f$ and its {\em polar}
\[
  f^\circ(y):=\inf\set{\mu>0}{\ip x y \le \mu f(x) \mbox{ for all } x}.
\]
When $f$ is a norm, the polar $f^\circ$ is simply the familiar dual
norm.  There is a close relationship between gauges, their polars, and
the support functions of their sublevel sets, as described by the
 identities~\cite[Prop.~2.1(iv)]{FriedlanderMacedoPong:2014}
\[
  f^{\circ} = \delta^{\star}_{[f\le1]}
  \quad\hbox{and}\quad
  f^{\star}=\delta_{[f^{\circ}\le1]}.
\]
We apply these
identities to the quantities involving $\rho$ and $\varphi$ in the
expression for the dual representation $\Phi$ in \eqref{eqn:dual-rep},
and deduce 
\[
  \delta^{\star}_{[\varphi\le\tau]}
  = \tau\delta^{\star}_{[\varphi\le1]}
  = \tau\varphi^\circ
  \quad\hbox{and}\quad
  \rho^{\star} = \delta_{[\rho^\circ\le1]}.
\]
Substitute these into $\Phi$ to obtain the equivalent
expression
\[
  \Phi(y,\tau) = \ip b y
  - \delta_{[\rho^\circ\le1]}(-y)
                 - \tau\varphi^\circ(A^T y).
\]
We can now write an explicit dual for the level-set problem~\ref{eqn:sp}:
\begin{equation}
  \label{eq:3}
  \maximize{y}\quad\ip b
  y-\tau\varphi^{\circ}(A^{T}y)
  \quad\st\quad
  \rho^{\circ}(-y)\le1.
\end{equation}
In the last three rows of the table, we set
$\rho=\norm{\cdot}_2$, which is self polar. For BPDN, we use the vector 1-norm
$\varphi=\norm{\cdot}_1$, whose polar is the dual norm
$\varphi^\circ=\norm{\cdot}_\infty$. For matrix completion, the function
$\varphi=\norm{\cdot}_*:=\sum_{i=1}^{\min\{m,n\}}\sigma_i(\cdot)$ is
the nuclear norm of a $n$-by-$m$ matrix, which is polar to the spectral
norm $\varphi^\circ= \sigma_{\submax}(\cdot)$. For the sharp elastic net,
we use Lemma~\ref{lem:sum-gauges} to deduce
\[
 (\alpha\norm{\cdot}_1 + \beta\norm{\cdot}_2)^\circ
 = (\gamma_{\frac1\alpha\bB_1} + \gamma_{\frac1\beta\bB_2})^\circ
 = \gamma_{\left(\frac1\alpha\bB_1\right)^\circ + \left(\frac1\beta\bB_2\right)^\circ}
 = \gamma_{\alpha\bB_\infty+\beta\bB_2}.
\]

A distinctive feature of all of the problems stated in
Table~\ref{tab:ns_reg} is the nondifferentiability of the objective
of~\ref{eqn:sp}. The choice seems especially peculiar when $\rho$ is
the 2-norm, since in that case, it is obvious that an
\emph{equivalent} smooth problem can be obtained by simply
squaring the objective \ref{eqn:sp} and the corresponding constraint
in the original problem \ref{eqn:ps}. Of course, we do not prescribe
the method for solving the level-set problem, and depending on the
application and solvers available, it may be more convenient or
efficient to solve a smooth variant of \ref{eqn:sp} in order to obtain
a solution of the nonsmooth version; cf. \S\ref{sec:ls-misfit}.

\begin{table}[t]
\centering\small
\begin{tabular}{l@{}l@{}l@{}l}
  \toprule
  Problem
& \multicolumn{1}{c}{\Psig}
& \multicolumn{1}{c}{\Qtau}
& \multicolumn{1}{c}{$\partial_{\tau} \Phi(y,\tau)$}
\\\midrule
 $\begin{array}[t]{@{}ll@{}}
    \rm gauge
  \\\rm optimization
  \end{array}$
  & $\begin{array}[t]{cl@{}}
      \displaystyle\min_{x} & \varphi(x)
    \\\textrm{s.t.} & \rho(Ax-b)\le\sigma
      \end{array}$
  & $\begin{array}[t]{cl@{}}
      \displaystyle\min_{x} & \rho(Ax-b)
    \\\textrm{s.t.} & \varphi(x)\le\tau
      \end{array}$
  & $-{\varphi^{\circ}(A^T y)}$
\\\midrule
     BPDN
    & $\begin{array}[t]{cl@{}}
         \displaystyle\min_{x} & \|x\|_{1}
       \\\textrm{s.t.} & \|Ax-b\|_{2}\le\sigma
         \end{array}$
    & $\begin{array}[t]{cl@{}}
         \displaystyle\min_{x} & \|Ax-b\|_{2}
       \\\textrm{s.t.} & \|x\|_{1}\le\tau
         \end{array}$
    & ${-\|{A^T y\|_\infty}}$
\\\midrule
$\begin{array}[t]{@{}ll@{}}
    \hbox{sharp}
  \\\hbox{elast-net}
  \end{array}$
  & $\begin{array}[t]{cl@{}}
       \displaystyle\min_{x} & \alpha\|x\|_{1}+\beta\|x\|_{2}
     \\\textrm{s.t.} & \|Ax-b\|_{2}\le\sigma
       \end{array}$
  & $\begin{array}[t]{cl@{}}
       \displaystyle\min_{x} & \|Ax-b\|_{2}
     \\\textrm{s.t.} & \alpha\|x\|_{1}+\beta\|x\|_{2}\le\tau
       \end{array}$
  & $-{\gam_{\alpha\bB_{\infty}+\beta\bB_{2}}(A^{T}\!y)}$
\\\midrule
 $\begin{array}[t]{@{}ll@{}}
    \rm matrix
  \\\rm completion
  \end{array}$
  & $\begin{array}[t]{cl@{}}
      \displaystyle\min_{X} & \|X\|_{*}
    \\\textrm{s.t.} & \|\cA X-b\|_{2}\le\sigma
      \end{array}$
  & $\begin{array}[t]{cl@{}}
      \displaystyle\min_{x} & \|\cA X-b\|_{2}
    \\\textrm{s.t.} & \|X\|_{*}\le\tau
      \end{array}$
  & $-{\sigma_1(\cA^{*}y)}$
\\\bottomrule
\end{tabular}
\caption{Nonsmooth regularized data-fitting. 
  }
\label{tab:ns_reg}
\end{table}

\subsection{Generalized linear models}
\label{subsec:GLM and robust regression}

In all the examples we have seen so far, we have encountered only two
types of misfit functions $\rho$, namely the squared 2-norm and the
various gauges listed in Table~\ref{tab:ns_reg}.
In this section, we broaden the scope by exploring several examples
arising from statistical modeling.  In particular, we consider the
broad class of {\it generalized linear models} (GLMs)~\cite{McNel},
which capture non-Gaussian data---including non-negative, count,
boolean and multinomial variables---and robust log-concave densities.

GLMs assume that the observed data is distributed according to a
member of the exponential family, and postulate a linear predictive
model for key parameters.  Suppose we are given data pairs
$\{(b_i,a_i)\}^n_{i=1}\subset \mathbb{R}\times\mathbb{R}^m$, where
$b_i$ is an observation associated with the covariate vector $a_i$ for
individual $i=1,\ldots,n$. GLMs assume that the postulated density for
each response $b_i$ is the function
\begin{equation}\label{glm with c}
  p(b_i;\, \theta_i) = C(b_i, \phi)\exp\left(\frac{b_i\theta_i - c(\theta_i)}{\phi}\right),
\end{equation}
where $\phi$ is the dispersion parameter, $\theta_i$ is the mean
parameter, $c(\cdot)$ is a function that specifies the distribution,
and $C(b_i,\phi)$ is a normalization constant that can depend on the data and
$\phi$.  To simplify the exposition, we focus only on the canonical
parameter $\theta_i$, and assume that the dispersion parameter $\phi$
is known and present in its simplest form; see McCullah and
Nelder~\cite{McNel} for more general cases.  Whenever the function $c$
is convex, it is clear that the resulting density is log-concave.  To
complete the GLM specification, one now assumes that $b_i$ is
distributed according to the GLM in~\eqref{glm with c} with
\[
\theta_i = a_i^Tx,
\]
where $x$ is an unknown vector that is uniform across the population
from which the data is selected. The task is to infer the vector $x$
from the given data.

A technical concept in GLM modeling is the {\it link function}---an
invertible function that maps likelihood parameters to the canonical
parameter $\theta_i$. For example, when working with count data, one
encounters the Poisson distribution, which is proportional to
$\exp(b_i\log \lambda_i - \lambda_i)$.  We identify~\eqref{glm with c}
with this distribution using the log link function, and set
$\theta_i = \log\lambda_i$. It necessarily follows that
$c(\theta_i) = \exp(\theta_i)$.

Assuming that the data is chosen independently from the population,
the negative log-likelihood function for this model is given by
\[
  L(b;Ax)=\sum^n_{i=1} -{\ln p(b_i; a_i^Tx)},
\]
where $a_i$ the $i$th row of the matrix $A$.  The
likelihood-constrained formulation \Psig\ for the regularized
GLM is thus given by the problem
\begin{equation}
  \label{eq:4}
  \minimize{x}\quad\varphi(x)\quad\text{subject to}\quad L(b;Ax)\le\sigma,
\end{equation}
where $\varphi$ is a given regularizer. For example, the 1-norm
regularizer may be used to induce sparsity in the parameter $x$.  A
reasonable choice for $\sigma$ is a proportion of the expectation:
\begin{equation}
\label{eq:genDist}
\sigma \propto \mathbb{E}\, L(b;\,Ax).
\end{equation}
When an estimate of the expectation is not available, $\sigma$ can be
selected by using an expected variance-reduction scheme, so that
$\sigma \propto L(b;0)$, where the proportionality constant is
chosen based on practitioner-prior experience.

\subsubsection*{Applying the level-set approach.}

We now describe the various ingredients needed to apply the level-set approach to
the GLM family.  For simplicity, we assume that $\varphi$ is a gauge,
which captures a broad range of regularizers
(cf.~\S\ref{sec:nonsmooth_reg}).  (Non-gauge regularizers are
considered in \S\ref{sec:mixed_norm}.) The corresponding level-set
problem \Qtau\ is given by
\begin{equation}
  \label{eq:5}
  \minimize{x}\quad L(b;Ax) \quad\st\quad \varphi(x)\le\tau.
\end{equation}
In order to derive global affine minorants, we require the
corresponding dual problem (cf.~\S\ref{sec:lower minorants via
  duality}). Set $L_b(\cdot):=L(b;\cdot)$, and apply Fenchel duality
to obtain
\begin{equation}
  \label{eq:6}
  \maximize{y}\quad - L_b^{\star}(-y) - \tau\varphi^{\circ}(A^{T}y).
\end{equation}
When $p$ is as given in~\eqref{glm with c}, we have
$L_b(z)=K+\sum_i\phi^{-1}(c(z_i)-b_iz_i)$, where
$K:=-\sum_i \ln C(b_i;\phi)$.  Hence, the dual problem takes the form
\begin{equation}
  \label{eq:7}
  \maximize{y}\quad K - \frac{1}{\phi} \sum_i c^{\star}(b_i-\phi  y_i)- \tau\varphi^{\circ}(A^{T}y).
\end{equation}
Table~\ref{table:glmBasic} lists common exponential distributions and
the link functions needed to represent them in the form of a
GLM~\eqref{glm with c}. The table also lists the resulting functions
$c$ and their conjugates needed for the dual.

\begin{table}[t]
\centering\small
\begin{tabular}{l@{\hskip 1cm}c@{\hskip 1cm}c@{\hskip 1cm}c}
  \toprule
   \multicolumn{1}{@{}c}{Distribution} 
 & $c(\theta)$ 
 & link function
 &  $c^{\star}(z)$
\\  \toprule
Gaussian
    & $\frac{1}{2}\theta^2$ 
    & $\rm{id}$
    &  $\frac{1}{2}z^2 $ 
\\
Huber~\cite{JMLR:aravkin13a}
    & $\rho_\kappa(\theta)$ 
    & $\rm{id}$
    & $\frac{1}{2}z^2 + \delta_{\kappa\bB_\infty}(z)$ 
\\
Poisson 
& $\exp(\theta)$
& $\log$
& $z\log z - z + \delta_{\bR_+}(z)$
\\
Bernoulli
& $\log(1+\exp(\theta))$ 
&$\rm{logit}$
& $z\log z + (1-z)\log(1-z) + \delta_{[0,1]}(z)$
\\
Gamma
& $-{\log}(-\theta)$ 
&$(\cdot)^{\ \mathclap{-1}}$
&$1-{\log}(-{z}) + \delta_{\bR_-}(z)$
\\\bottomrule
\end{tabular}
\caption{Parameters of the GLM family, including required conjugates
  for their dual representation. The interpretation of  coercive 
  PLQ penalties (such as the Huber) as kernels of statistical 
  distributions is developed in~\cite[Section 2]{JMLR:aravkin13a}.}\label{table:glmBasic}
\end{table}

\subsubsection{A fair comparison of regularizers}
\label{sec:learning}

Multiple experiments that involve different regularization functions
can be easily compared at the same admissible levels of misfit using the
formulation \Psig. This feature of \Psig\ is unique
among the alternative formulations.

As an example, consider classification using logistic regression
(corresponding to the Bernoulli distribution) with either 1- or 2-norm
regularization:
\begin{equation}
  \label{eq:lrForm}
  \minimize{x}\quad\|x\|_i \quad\st\quad L(b;Ax) \leq  \sigma,
\end{equation}
for $i = 1,2$. We set $\sigma:=L(b;0)/\eta$, where $\eta$ is a
specified proportionality constant. 
The likelihood of observing a Bernoulli random variable $b_i \in \{0, 1\}$ is 
given by 
\[
P(b_i) = \eta_i^{b_i} (1-\eta_i)^{1-b_i},
\]
where $\eta_i$ is the probability of observing $b_i = 1$. 
Rewriting to match~\eqref{glm with c}
gives 
\[
\begin{aligned}
P(b_i) & = \exp\left( b_i \log(\eta_i) + (1-b_i) \log(1-\eta_i)\right)\\
& = \exp\left( b_i \log\left(\frac{ \eta_i}{1-\eta_i}\right) +  \log(1-\eta_i)\right),
\end{aligned}
\]
which identifies the link function from Table~\ref{table:glmBasic}
with the canonical parameter
$\theta_i = \log\left(\frac{ \eta_i}{1-\eta_i}\right)$, and determines
$c(\theta_i) = \log(1+\exp(\theta_i))$. Composing with the linear
model $\theta_i =  a_i^T x$, we obtain the negative log
likelihood objective (ignoring the constant term)
\[
L(b; Ax) = \sum_{i=1}^n \log\left(1 + \exp( a_i^T x )\right) - b_i (a_i^T x).
\]

\begin{table}[tb]
  \centering
\begin{tabular}{lccccc}
\toprule
\multicolumn{1}{c}{$\eta$} & $1.1$ & $1.5$ & $1.9$ & $2.0$ & $2.1$
\\ \midrule
2-norm correct $+$ &0&0.07&0.46&0.52&0.57\\ 
1-norm correct $+$ &0&0.17&0.51&0.52&0.57\\ 
2-norm correct $-$ &.96&0.98&0.94&0.94&0.93\\ 
1-norm correct $-$ &.96&0.98&0.94&0.94&0.93\\ 
2-norm nonzero features &89&116&112&122&122\\ 
1-norm nonzero features &1&5&16&22&42\\ \bottomrule
\end{tabular}
\caption{Recovery results for likelihood-regularized \Psig\ logistic 
  regression formulations. Fractions of correctly identified cases and controls in {\it test} set are
  shown for classifiers corresponding to optimal 2-norm and 1-norm solutions of~\eqref{eq:lrForm}, fitting the data in terms of the proportionality constant 
  $\eta$.} \label{tab:lr} 
\end{table}
We run the approach on the Adult dataset~\cite{Lichman:2013}, which
aims to predict whether people make more than \$50K a year. The
challenge is that there are fewer positive than negative answers. The
full dataset has $m=122$ features and 48,844 individuals. We split this
group into $n=32562$ training and 16,282 test cases.  In the test set,
there are 3,846 individuals who make more than \$50K a year, and 12,436
who do not.
Table~\ref{tab:lr} shows that the 1-norm regularization 
has as good or better generalizability at all tested levels of $\eta$. 
The 1-norm does as well or better than 2-norm with the cases (people
earning more than \$50K), and gives a sparser model, while matching
identification of controls (people earning less than \$50K).

\subsubsection{Robust regression}

As another example, we consider log-concave robust penalties---an important subclass of GLMs.  We
illustrate the modeling possibilities of this subclass, using the
Huber penalty and its asymmetric extension, the quantile Huber (see
Figure~\ref{fig:hubs}).  The quantile Huber is parameterized by
$(\kappa, \tau)$, which control the transition between quadratic and
linear pieces, as well as the asymptotic slopes:
\begin{eqnarray}
\rho_{\kappa,\tau}(r) 
\label{quantileHuber}
&=& \begin{cases}
\tau |r| - \frac{\kappa \tau^2}{2} & \text{if} \text{ } r < -\tau\kappa,\\
\frac{1}{2\kappa}r^2 & \text{if} \text{ } r\in [-\kappa \tau, (1-\tau)\kappa],\\ 
(1-\tau) |r| - \frac{\kappa(1-\tau)^2}{2} & \text{if} \text{ } r > ~ (1-\tau)\kappa.
\end{cases}
\end{eqnarray} 
The quantile Huber generalizes both the quantile loss and the Huber
loss. We recover Huber when $\tau = 0.5$, and the quantile Huber
converges to the quantile loss (known as the \emph{check function}) as
$\kappa \rightarrow 0$. When $r$ is an $m$-vector instead of scalar,
we write $\rho_{\kappa,\tau}(r):=\sum_{j=1}^m\rho_{\kappa,\tau}(r_j)$,
and for simplicity we write $\rho_\kappa:=\kappa\rho_{2\kappa,0.5}$ to
denote the scaled Huber.

The Huber penalty figures prominently
 in high-dimensional regularized 
{\em robust} regression, as a %
measure of data misfit \cite{Hub,Mar,Bube2007,Dutt,Clark85,LiW98}.
High dimensional extensions (with sparse regularization) have been
studied by Sun and Zhang~\cite{sun2012scaled} with applications to
face recognition~\cite{yang2011robust} and signal
processing~\cite{kekatos2011sparse}. %
The quantile Huber, shown in Figure~\ref{fig:HuberQ}, was recently
introduced by Aravkin et al.~\cite{aravkin2014qh} as an alternative to
quantile regression---an asymmetric variant of the 1-norm used to
analyze heterogeneous datasets~\cite{KB78,Buchinsky:1994}, such as
those in computational biology~\cite{Zou08}, survival
analysis~\cite{KG01}, and
economics~\cite{KH01,Koenker:2005}. %

The methods of \S\ref{algo:newton} allow one to easily explore robust
regularization with the Huber penalty in the context of
sparsity. Specifically, consider the \BPsig\ problem, but with the
Huber penalty replacing the norm-squared error:
\begin{equation}\label{eq:Huber problem}
\minimize x \quad\norm{x}_1 \quad\st\quad \rho_{\kappa,\tau}(b-Ax) \leq \sigma.  
\end{equation}
It is well known that the Huber loss function is much less sensitive
(i.e., robust) to outliers in the data than the norm-squared. %
\begin{figure*}[t]
\centering
	\begin{subfigure}[t]{0.4\textwidth}
		\centering
		\begin{tikzpicture}
		\begin{axis}[
		thick,
		height=3cm,
		xmin=-2,xmax=2,ymin=0,ymax=1,
		no markers,
		samples=50,
		axis lines*=left, 
		axis lines*=middle, 
		scale only axis,
		xtick={-1,1},
		xticklabels={},
		ytick={0},
		] 
		\addplot[red,domain=-2:-1,densely dashed]{-x-.5};
		\addplot[blue, domain=-1:+1]{.5*x^2};
		\addplot[red,domain=+1:+2,densely dashed]{x-.5};
		\addplot[blue,mark=*,only marks] coordinates {(-1,.5) (1,.5)};
		\end{axis}
		\end{tikzpicture}
		\caption{\label{fig:Huber}Huber, $\kappa = 1$}
	\end{subfigure}    
	\begin{subfigure}[t]{0.4\textwidth}
		\centering
		\begin{tikzpicture}
		\begin{axis}[
		thick,
		height=3cm,
		xmin=-2,xmax=2,ymin=0,ymax=1,
		no markers,
		samples=100,
		axis lines*=left, 
		axis lines*=middle, 
		scale only axis,
		xtick={-.24,.56},
		xticklabels={},
		ytick={0},
		] 
		\addplot[red,domain=-2:-2*0.3*0.4,densely dashed]{0.3*abs(x) - 0.4*0.3^2};
		\addplot[blue,domain=-2*0.3*0.4:2*(1-0.3)*0.4]{0.25*x^2/0.4};
		\addplot[red,domain=2*(1-0.3)*0.4:2,densely dashed]{(1-0.3)*abs(x) - 0.4*(1-0.3)^2};
		\addplot[blue,mark=*,only marks] coordinates {(-.24,0.0550) (0.56,0.20)};
		\end{axis}
		\end{tikzpicture}
		\caption{\label{fig:HuberQ} quantile Huber}
	\end{subfigure}  
	\caption{Huber penalty and its asymmetric extension, quantile Huber. \label{fig:hubs}}
\end{figure*}
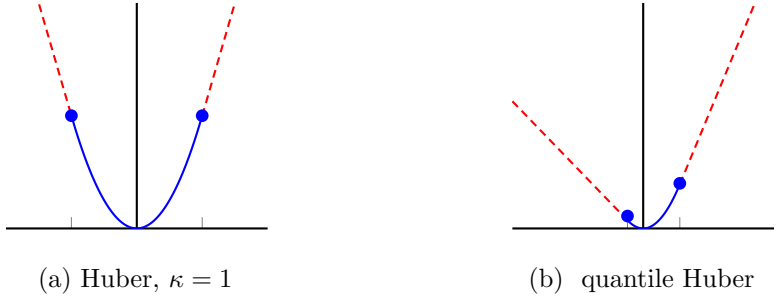

\begin{exa}[Robust sparse regression]\label{ex:rsr}
  As a proof of concept, we illustrate the level-set framework on the
  following example. We generate a $k$-sparse signal of dimension
  $n\gg k$, measure it with $m = 5k$ Gaussian random vectors, and
  contaminate the measurements with asymmetric outliers.  The results
  are shown in Figure~\ref{fig:HuberResults}. In the experiment,
  $n = 400$, $m = 100$, and $k=10$.  True measurements are obtained,
  and small Gaussian noise is added.  The measurements are then
  contaminated by six positive outliers generated by sampling
  uniformly from $[0, 0.5]$. The 2-norm, symmetric Huber, and quantile
  Huber are compared using our proposed level-set framework; all
  models are fit to a level $\sigma = 0.05\rho(b)$, where $b$ is the
  (contaminated) measurement vector. Both symmetric and quantile Huber
  show superior performance to the 2-norm. The advantage of the
  asymmetric Huber is fully evident in the residual plot. All the
  outliers in the example are positive, and using $\tau = 0.9$ for the
  quantile Huber, we identify all the outliers in the residual.

\begin{figure*}[h!]
   \begin{subfigure}[t]{0.47\textwidth}
     \includegraphics[width=\textwidth]{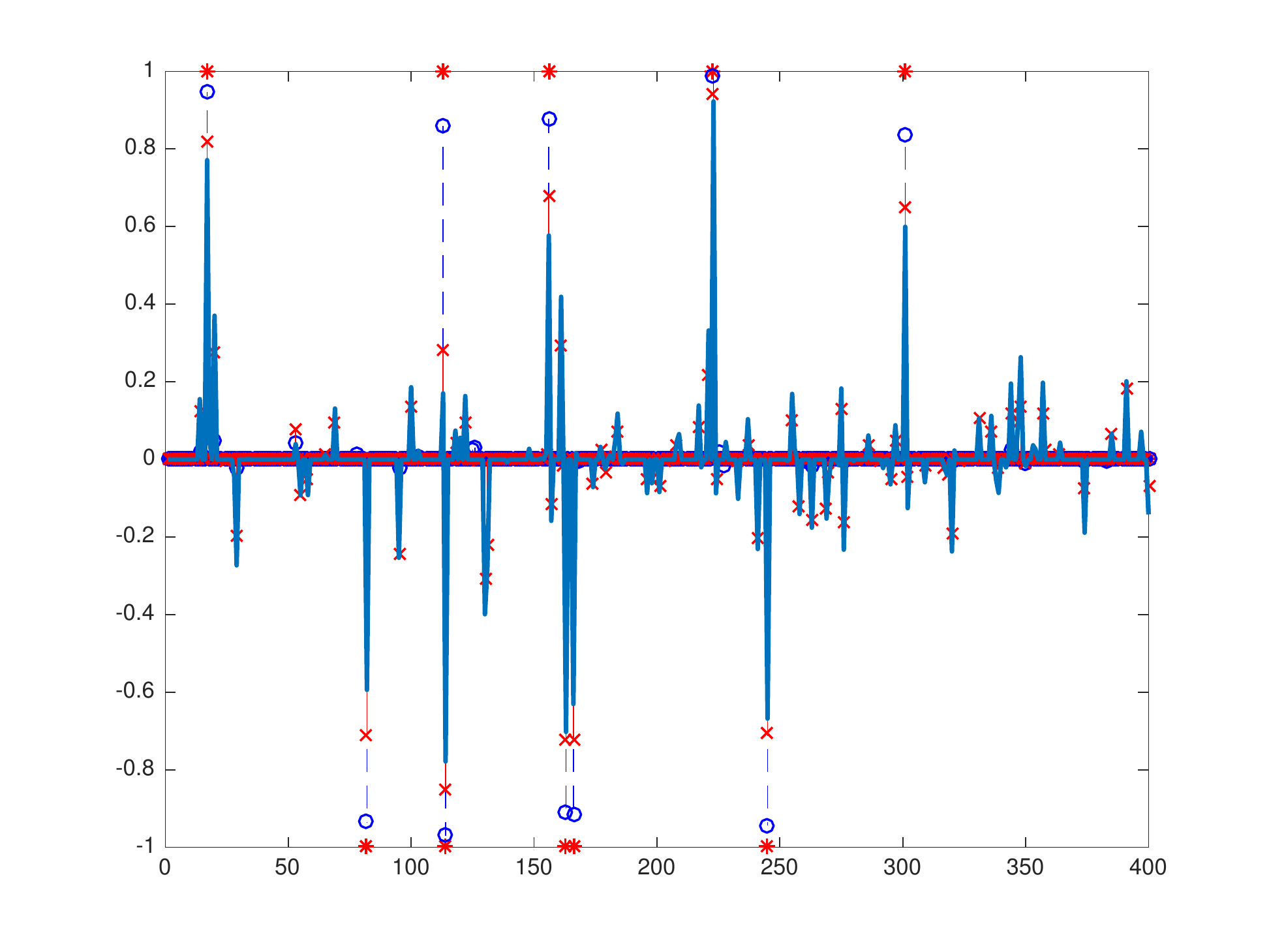}
 \caption{\label{fig:Signal} True and Fitted signals. Red asterisks show true sparse signal; blue solid line shows LS estimate; 
 solid dashed line ending in `x' shows Huber estimate; thin dashed line ending with `o' marker shows quantile Huber estimate. }
\end{subfigure}    
\hfill
   \begin{subfigure}[t]{0.47\textwidth}
     \includegraphics[width=\textwidth]{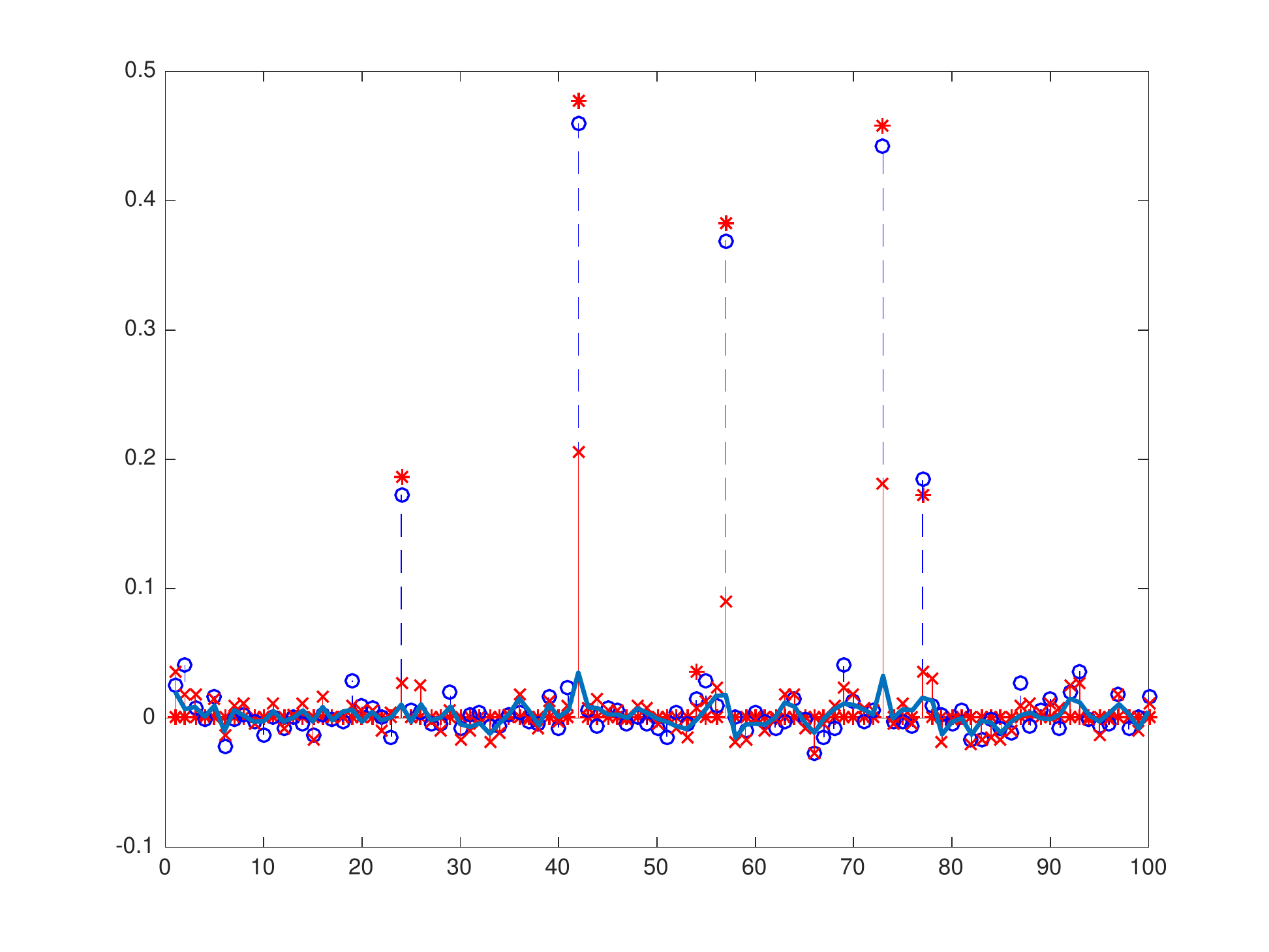}
\caption{\label{fig:Res} True and Fitted Residuals. Red asterisk shows true outliers; blue solid line shows LS residual; 
 solid dashed line ending in `x' shows Huber residual; thin dashed line ending with `o' marker shows quantile Huber residual. }
\end{subfigure}  
    \caption{\label{fig:HuberResults}Robust Asymmetric Recovery, comparing 2-norm, Huber, and quantile Huber penalties. 
    $\kappa=0.1$ for both Huber penalties; quantile Huber has $\tau = 0.9$ 
    to capture the fact that outliers are expected to be positive. }
\end{figure*}
    
    \end{exa}

\section{Case studies}
\label{case:stud}

\subsection{Low-rank matrix completion}
\label{subsec:low-rank-matrix}

A range of useful applications can be modeled as matrix completion
problems. Important examples include applications in recommender
systems and system identification (Recht, Fazel,
Parillo~\cite{RecFazPar:10}). The general principle extends to robust
principal component analysis (RPCA), where we decompose a signal into
low rank and sparse components, and its variants, including its stable
version, which allows for noisy measurements.  Applications include
alignment of occluded images~\citep{PenGanWri:12}, scene
triangulation~\citep{ZhaLiaGan:11}, model selection~\citep{CPW12},
face recognition, and document indexing~\citep{CanLiMa:11}.

These problems can be formulated generally as
\begin{equation}\label{eq:mat-comp-generic}
  \minimize{X\in\bR^{m\times n}}\quad\varphi(X)
  \quad\st\quad\rho(\cA X- b)\le\sigma,
\end{equation}
where $b$ is a vector of observations, the linear operator $\cA$
encodes information about the measurement process, and the objective
$\varphi$ encourages the required structure in the solution, e.g.,
low-rank. The function $\rho$ measures the misfit between the linear
model $\cA X$ and the observations $b$. If we wish to require $\cA
X=b$, we can simply set $\sigma=0$ and choose any nonnegative convex function $\rho$ with
$\rho^{-1}(0)=\{0\}$, e.g., $\rho=\twonorm{\cdot}$.
We categorize the problems of interest into two broad classes:
\emph{symmetric} and \emph{asymmetric} problems. For each case, we
outline how the level-set approach leads to implementable algorithms
with computational kernels that scale gracefully with problem size.

The first class of problems aims to recover a low-rank PSD matrix, and
in that case, the linear operator $\cA$ maps between the space of
symmetric $n\times n$ matrices and vectors, and we define the objective $\varphi$ by
\[
  \varphi_1(X) = \tr(X) + \delta_{\mathcal{S}^n_+}(X).
\]
 Problem~\eqref{eq:mat-comp-generic} then reduces to finding a
minimum-trace, PSD matrix that satisfies the measurements specified by
$\cA X=b$. There are analogs for optimization over complex Hermitian matrices; we focus on the real case only for simplicity. The formulation above captures, for example, the
\emph{PhaseLift} approach to the phase-retrieval problem, which aims
to recover phase information about a signal (e.g., an image) by using
only a series of magnitude
measurements~\cite{candes2012phaselift}. Important applications include
optical wavefront reconstruction for astrophysical imaging~\cite{BLL02}
and the 
imaging of the molecular structure of a crystal via X-ray
crystallography, which gives rise to such magnitude-only measurements;
see Waldspurger, d'Aspremont, and Mallat~\cite{waldspurger2015phase}
for a more complete description, including a number of other
applications.

The second class of matrix-recovery problems does not require
definiteness of $X$. In this case, the linear operator
$\cA$ on $\mathbb{R}^{m\times n}$ is not restricted to symmetric
matrices, and we define $\varphi$ as the nuclear norm:
\[
  \varphi_2(X) = \nucnorm{X}
  := \sum_{i=1}^{\mathclap{\min\{m,n\}}}\sigma_i(X),
\]
where $\sigma_i(X)$ is the $i$th singular value of $X$. This
formulation captures, for example, the bi-convex compressed sensing
problem~\cite{ling:2015}.

\begin{exa}[Robust PCA]\label{ex:RPCA}
  The second class captures a range of problems that are not
  immediately of the form \eqref{eq:mat-comp-generic}.  For example,
  the stable version of the RPCA problem~\cite{RPCA_algo_Wright} aims
  to decompose a matrix $Y$ as a sum of a low-rank matrix and a sparse
  matrix via the problem
\begin{equation} \label{eq:sum-SPCP-QP} 
   \minimize{L,S}\quad \lambda \|L\|_* + \kappa \|S\|_1
     + \half\|\mathcal{A}(L-Y)-S\|_F^2.
\end{equation}
Here the operator $\mathcal{A}$ is often a mask  for the known elements of $Y$.
The goal is to obtain a low-rank approximation to $Y$ where the deviation from the
known elements of $Y$ is as sparse as possible.
The parameters $\lambda$ and $\kappa$ are are chosen to balance
the rank of $L$ against the sparsity of the residual $S$ while
minimizing the least-squared misfit.  
This model can be given a statistical interpretation that fits nicely into the context
of robust regression as presented in Section \ref{subsec:GLM and robust regression}.

We proceed by eliminating $S$ in \eqref{eq:sum-SPCP-QP} by first minimizing
the objective over $S$ alone---an overlooked algorithmic technique for this problem. 
Observe that, as a function of $S$, the objective 
is the Moreau
envelope of the 1-norm evaluated at $\cA(L-Y)$, or, equivalently,
the Huber function $\rho_\kappa$ 
on $\bR^{m\times n}$~\eqref{quantileHuber}:
\[
\begin{aligned}
\inf_S \left\{\kappa \|S\|_1\! + \half\|\mathcal{A}(L-Y)\!-\!S\|_F^2\right\} 
=\! \rho_{\kappa} (\mathcal{A}(L\!-\!Y)).
\end{aligned}
\]
Problem~\eqref{eq:sum-SPCP-QP} can now be written in terms of $L$
alone:
\begin{equation*} \label{eq:sum-SPCP-QP-RED} 
   \minimize{L}\quad \lambda \|L\|_* + \rho_{\kappa }(\mathcal{A}(L-Y)).
\end{equation*}
This is the Lagrangian form of the robust estimation problem
\eqref{eq:Huber problem}. 
Arguably, we can now interpret the goal of this problem as one
of finding the lowest rank approximation to $Y$ over its known elements
subject to a bound on a robust measure of misfit. This yields the problem
\begin{equation} \label{eq:sum-SPCP-RED} 
   \minimize{L}\quad \nucnorm{L} \quad \st\quad \rho_{\kappa}(\mathcal{A}(L-Y)) \leq \sigma,
\end{equation}
for some choice  of parameter $\sigma\ge 0$. 
Various principled choices for
$\sigma$ are discussed in \S\ref{subsec:GLM and robust regression}.
\end{exa}

\subsubsection*{Level-set approach and the Frank-Wolfe oracle}

We apply the level-set approach, and exchange the roles of the
regularizing function $\varphi$ and the misfit $\rho(\cA X-b)$. Note
that the objective function $\varphi_1$ for the symmetric case
vanishes at the origin, and is convex and positively homogeneous It is
thus a gauge.  The second objective function $\varphi_2$ is simply a
norm. Therefore, for both cases, we may use the first row of
Table~\ref{tab:ns_reg} to determine the corresponding level-set
subproblem and affine minorants based on dual certificates. In particular, the
corresponding level-set subproblem \ref{eqn:sp}, which defines the
value function, is
\[
v(\tau) :=
     \min_{X}\set{ \rho(\cA X-b)}{\varphi(X) \le \tau}.
\]
We use the polar calculus described by Friedlander et
al. \cite[\S7.2.1]{FriedlanderMacedoPong:2014} and the definition of
the dual norm to obtain the required polar functions
\[
\varphi_1^\circ(Y) = \max\{0,\ \lambda_1(Y) \}
\textt{and}
\varphi_2^\circ(Y) = \sigma_1(Y)
\]
for the symmetric and asymmetric cases, respectively.

The evaluation of the affine minorant oracle requires an approximate
solution of the optimization problem that defines the value function
$v$, and computation of either an extreme eigenvalue or singular value
to determine an affine minorant. As numerous authors have observed, the Frank-Wolfe
algorithm \cite{jaggi:2013,FW-alg} is therefore especially well suited for 
evaluating the required quantities, and here we describe how to apply
the algorithm to this setting.

The Frank-Wolfe subproblem \eqref{eqn:FW}, used to generate search
directions at each iteration, takes the form
\begin{equation}\label{eq:fw-subprob}
  \maximize{S}\ \ip{G}{S}
  \quad\st\quad \varphi(S)\le\tau.
\end{equation}
where $G:=\cA^*\nabla\rho(\cA X-b)$ is the gradient of $\rho(\cA X-b)$
evaluated at the current primal iterate $X$. Note that the steplength
in this case is easily obtained as the minimizer of the quadratic
objective along the intersection of $[\varphi\leq \tau]$ and the ray
$X+\mathbb{R}_+(S-X)$.

Solutions for the linearized subproblems can be obtained by computing
extreme eigenvalues or singular values of
$G$~\cite[\S4.2]{jaggi:2013}. For the symmetric case, the constraint
\[
\varphi_1(S)\le\tau
\textt{is equivalent to}
\tr(S)\le\tau,\ S\succeq0.
\]
The linearized subproblem~\eqref{eq:fw-subprob} is then solved by any
matrix of the form
\[
  S = U \Diag(\xi_i) U^T
  \quad\hbox{with}\quad
  \sum_{i=1}^{k}\xi_i=\tau,\ \xi_i\ge0,
\]
where $U\in\mathbb{R}^{n\times k}$ is the matrix that collects the $k$
eigenvectors of $G$ corresponding to
$\lambda_1(G)$.  For the non-symmetric case, the constraint
$\varphi_2(S)\le\tau$ is simply $\nucnorm{S}\le\tau$, and the
linearized subproblem is solved by any matrix of the form
\[
  S = U\Diag(\xi_i)V^T
  \textt{with}
  \sum_{i=1}^{k}\xi_i=\tau, \ \xi_i\ge0,
\]
where $U\in\mathbb{R}^{m\times k}$ and $V\in\mathbb{R}^{n\times k}$
are the matrices that collect the $k$ singular vectors of $G$
corresponding to the leading singular value $\sigma_1(G)$.  In both
cases, Krylov-based eigensolvers, such as
ARPACK~\cite{lehoucq1998arpack} can be used for the required
eigenvalue and singular-value computation. If matrix-vector products
with the matrix $\mathcal{A}^*y$ and its adjoint are computationally
inexpensive, the computation of a few rightmost eigenvalue/eigenvector
pairs (resp., maximum singular value/vector pairs) is much cheaper
than the computation of the entire spectrum, as required by a method
based on projections onto the feasible region. Such circumstances are
common, for example when the operator $\mathcal{A}$ is sparse or it is
accessible through a Fast Fourier Transform (FFT). The following
example illustrates exactly this scenario.

\begin{exa}[Euclidean distance completion]
\label{ex:edcp}

A common problem in distance geometry is the inverse problem: given
only local pairwise Euclidean distance measurements  among a set of
points, recover their location in space.
Formally, given a weighted
undirected graph $G=(V,E,\omega)$ with a vertex set
$V=\{1,\ldots,n\}$, and a target dimension $r$, the {\em Euclidean
  distance completion problem} asks to determine a collection of
points $p_1,\ldots, p_n$ in $\bR^r$ approximately satisfying
$$
 \|p_i-p_j\|^2=\omega_{ij} \textt{for all edges} ij\in E.
$$
In literature, this problem is also often called $\ell_2$ graph
embedding and appears in wireless networks, statics, robotics, protein
reconstruction, and manifold learning; see the recent survey
\cite{surv}.

A popular convex relaxation for this problem was introduced by
Weinberger et al.~\cite{Weinberger}, and extensively studied by a
number of authors \cite{Bsen,bis_ye_toh_wang,noisy}:
\begin{equation}\label{eqn:maxtrace}
  \begin{array}{ll}
    \maxim & \tr X
 \\ \st    & \begin{aligned}[t]
                &\|\mathcal{P}_E\circ\mathcal{K}(X)-\omega\| \leq \sigma,
              \\&Xe=0,\ X\succeq 0,
             \end{aligned}
  \end{array}
\end{equation}
where $\mathcal{K}\colon \mathcal{S}^n\to\mathcal{S}^n$ is the mapping  $[\mathcal{K}(X)]_{ij}=X_{ii}+X_{jj}-2X_{ij}$
and $\mathcal{P}_E(D)$ is the canonical projection of a matrix $D$ onto entries indexed by the edge set $E$.
Indeed, if $X$ is a rank $r$ feasible matrix, we may factor it into $X=PP^T$, where $P$ is an $n\times r$ matrix. It is then easy to see that the rows of $P$ are the points $p_1,\ldots,p_n\in \mathbb{R}^r$ we seek. The constraint 
$Xe=0$ simply ensures that the points $p_i$ are centered around the origin.
Notice, that this formulation directly contrasts the usual min-trace regularizer in compressed sensing; nonetheless, it is very natural. An easy computation shows that in terms of any factorization $X=PP^T$, the equality $\tr(X)= \frac{1}{2n}\sum^n_{i,j=1} \|p_i-p_j\|^2$ holds. Thus trace maximization serves to ``flatten'' the realization of the graph. 

It is known that for $\sigma=0$, the problem formulation
\eqref{eqn:maxtrace} notoriously fails strict feasibility
\cite{noisy,coord, nath_sen}. In particular, for small $\sigma\geq 0$
the feasible region is very thin and the solution to the problem is
unstable. As a result, algorithms maintaining feasibility are likely
to exhibit some difficulties. In contrast, following the theme of this
paper, we employ an {\em infeasible} method, and hence the poor
conditioning of the underlying problem does not play a major role.
The least-squares level-set problem that corresponds to the
minimization formulation of \eqref{eqn:maxtrace} is
\begin{equation}\label{eqn:complex_bd}
  \begin{aligned}
    \begin{array}[t]{ll}
      \minim & \|\mathcal{P}_E\circ\mathcal{K}(X)-\omega\|
    \\\st & \tr X\geq \tau,\ Xe=0,\  X\succeq 0.
    \end{array}
  \end{aligned}
\end{equation}
Note the direction of the inequality $\tr X\geq \tau$, which takes
into account that the original formulation \eqref{eqn:maxtrace} is a
\emph{maximization} problem. As a result, the root-finding method on
the value function will approach the optimal value $\tau_*=\OPT$ from
the right. In particular, to initialize the approximate Newton scheme,
we need an upper bound $\tau_0$ on the objective function. Such upper
bounds are easily available from the diameter of the graph. See
Figure~\ref{fig:prob1_6_2} for an illustration.

\begin{figure}[t]
  \centering
  \centering
  \includegraphics[trim=0cm 1cm 0cm 0.5cm,scale=0.65]{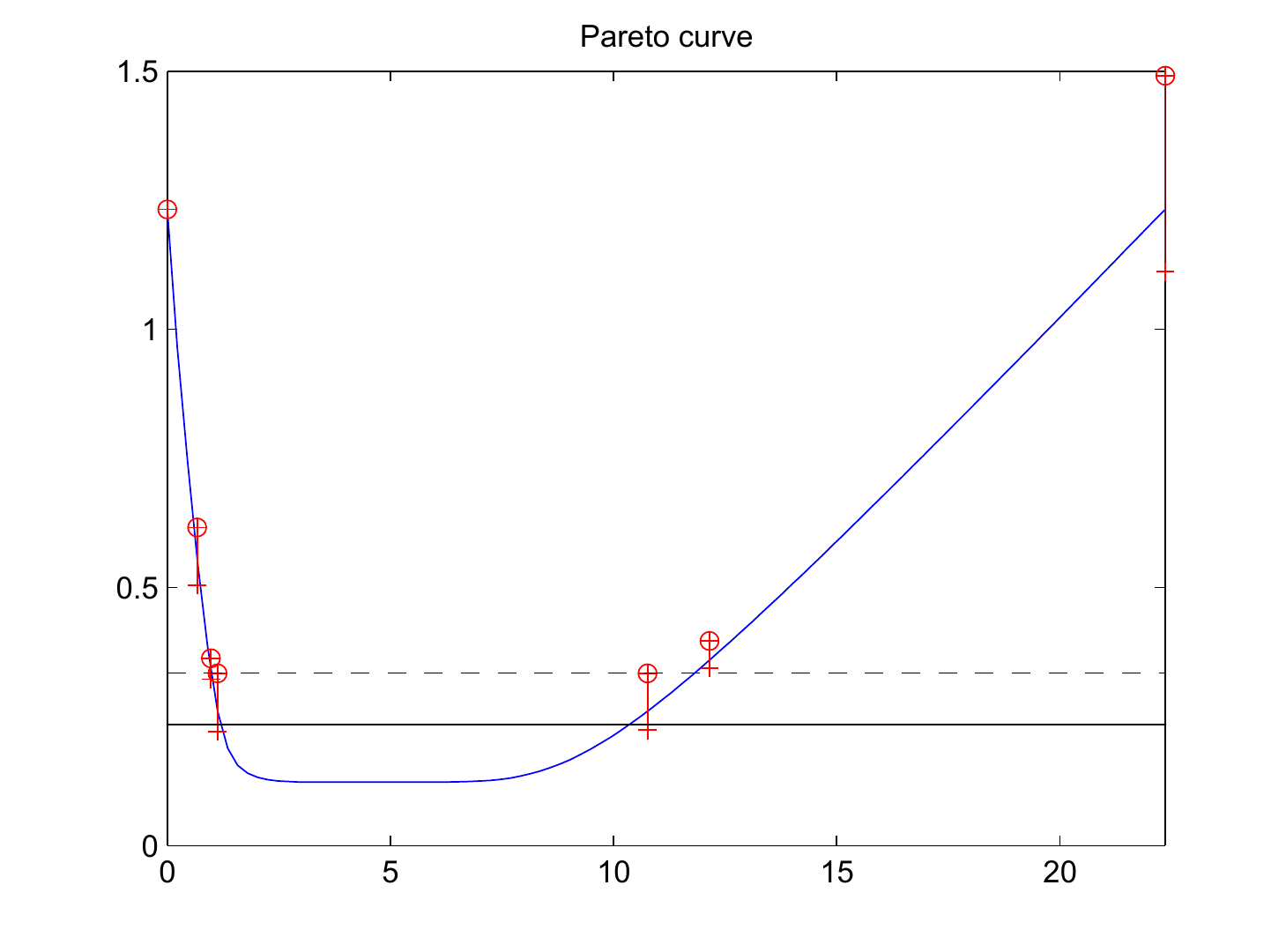}
  \caption{The value function
    $v(\tau):=\inf\set{\|\mathcal{P}_E\circ\mathcal{K}(X)-\omega\|}{
      \tr X= \tau,\ Xe=0,\ X\succeq 0}$.
    Newton's method converges to either the minimum- or maximum-trace
    solution, depending on if it is started with an iterate to the
    left of the minimal root, or to the right of the maximal
    root. For a solution of~\eqref{eqn:maxtrace}, we require the
    maximal root.  In this experiment, $\sigma =0.25$.}
  \label{fig:prob1_6_2}
\end{figure}

Note that the gradient of the
objective function is typically very sparse (as sparse as the edge set
$E$). Moreover, the linear subproblem over the feasible region is analogous to the ones considered in Section~\ref{subsec:low-rank-matrix}, requiring only a maximal eigenvalue computation on a sparse matrix (the gradient of the objective function); for more details see \cite{noisy}.
This makes the problem \eqref{eqn:complex_bd}
ideally suited for the Frank-Wolfe algorithm, as discussed in \S \ref{subsec:low-rank-matrix}.
We note that the dual problem of \eqref{eqn:complex_bd} takes the form
\[
  \maximize{y\in \bR^E,\ \|y\|_2\leq 1}
  \quad \langle y,\omega \rangle
  -2\tau\lambda^{e^{\perp}}\submax(\Diag(Ye)-Y).
\]
The matrix $Y=\mathcal{P}^*_E(y)$ is the vector $y$ padded with zeros
and then  $2(\Diag(Ye)-Y)=\mathcal{K}^*\mathcal{P}^*_E(y)$. The
symbol $\lambda^{e^{\perp}}\submax(A)$ is the maximal eigenvalue of the
restriction of the matrix $A$ to $e^{\perp}$. Hence, affine minorants are
immediate to read off from the dual certificates generated by the
Frank-Wolfe algorithm.
An extensive numerical investigation of this approach is made by
Drusvyatskiy et al.~\cite{noisy}.
\end{exa}

\subsection{Robust elastic net regularization}
\label{sec:mixed_norm}
In this final section, we explore an important data fitting problem
where the regularizer $\varphi$ is not a gauge, unlike our previous
examples. Zou and Hastie~\cite{EN_2005} introduced the {\em elastic
  net regularizer}
\[
\vpen(x):=\alf\onorm{x}+\frac{1-\alf}{2} \tnorm{x}^2\qquad (0\le \alf\le 1)
\]
for situations where there are multiple groups of covariates that are
strongly correlated within each group.  In this setting, the LASSO
typically picks one member from each of the most important groups
whereas the elastic net can pick out both the important groups and
their members.  As is the case with the Huber function, $\vpen$ is a
member of the PLQ family~\cite{JMLR:aravkin13a,
  AravkinBurkeFriedlander:2013}.

Zou and Hastie only consider the \LStau\ and \QPlam\ formulations of
the 1-norm regularized problem discussed in \S\ref{sec:intro}, but
with $\onorm{\cdot}$ replaced by $\vpen$. Furthermore, they focus on
the Lagrangian formulation \QPlam\ for computational reasons. The
problem corresponding to \BPsig\ is not investigated.  In this
section, we provide a guide to the implementation of the methods of
\S\ref{sec:alg_fram} for this version of the elastic net problem, but
generalized to the case where the residual term is replaced by the
Huber function $\rho_{\kappa}$ in~\eqref{quantileHuber} for robust
inference.
This gives the three formulations described in Table
\ref{tab:elastic-net}, which we call the \emph{robust elastic net
  problem}.

\begin{table}[tb]
\small
\centering
\begin{tabular}{lll}
\toprule
   \multicolumn{1}{c}{\Psig}
 & \multicolumn{1}{c}{\Qtau}
 & \multicolumn{1}{c}{Dual of \Qtau}
\\  \midrule
$\begin{array}[t]{@{}c@{\ }l}
         \displaystyle\min_{x} & \vpen(x) %
         \\\textrm{s.t.} & \rho_\kappa({Ax-b})\le\sigma %
        \end{array}$
& $\begin{array}[t]{c@{\;}l}
         \displaystyle\min_{x} & \rho_\kappa(Ax-b) %
       \\\textrm{s.t.} &  \vpen(x)\le\tau %
         \end{array}$
    &  $\begin{array}[t]{c@{\;}l@{}}
         \displaystyle\max_{y\in \kappa\bB_\infty} & \ip{b}{y}-\frac{\kappa}{2}\tnorm{y}^2
         \ -\del^{\star}_{[\vpen\le\tau]}(A^Ty) %
         \end{array}$
\\\bottomrule
\end{tabular}
\caption{Elastic net \label{tab:elastic-net}}
\end{table}

\subsubsection*{Inexact oracle for the value function}
From Table~\ref{tab:elastic-net}, we determine the value function
\[
v(\tau):= \min\set{\rho_\kappa(Ax-b)}{\vpen(x) \le\tau}
\]
to which we apply the root-finding procedure.  We solve \Qtau\ via an
optimal gradient-projection algorithm, as described in
\S\ref{sec:least-squares-level}. The methods require at each iteration
a projection onto the level sets $[\vpen\le \tau]$, which is given as
the solution of the problem
\begin{equation}\label{project en}
  \minimize{x} \quad \half \tnorm{x-z}^2\quad\st\quad \vpen(x)\le \tau.
\end{equation}

The projection problem can be solved as follows. Assume without loss
of generality $z\notin [\varphi_{en}\le \tau]$ since otherwise $x:=z$
solves \eqref{project en}. We may also assume---possibly after a
coordinate sign change---that $z\geq 0$. Observe then that any optimal
solution $x$ satisfies $x\geq 0$.  Thus, a feasible point $x$
solves~\eqref{project en} if and only if there exists a scalar
$\lam> 0$ that satisfies
\[
0  \in (x - z) + \lambda(1-\alpha) x + \lambda\alpha \partial \|\cdot\|_1 (x).
\]
Equivalently,
\[
z  \in (1 + \lambda(1-\alpha))x + \lambda\alpha \partial \|\cdot\|_1 (x) ,
\]
which amounts to the coordinate-wise inclusion 
\[
  z_i \in (1 + \lambda(1-\alpha))x_i + \lambda\alpha \partial |\cdot|
  (x_i) \quad\mbox{for each $i=1,\dots, n$}.
\]
In the case $x_i=0$, simple arithmetic shows
$(z_i - \lambda\alpha \sign{(z_i)})_+=0$.  Otherwise when $x_i\neq 0$,
the numbers $x_i$ and $z_i$ are both strictly positive, and
\begin{equation}\label{two conditions 1}
x_i = \frac{(z_i - \lambda\alpha \sign{(z_i)})_+}{1+\lambda(1-\alpha)} = \frac{(z_i - \lambda \alpha)_+}{1+\lambda(1-\alpha)}.
\end{equation}
Hence, regardless of whether $x_i$ is zero or not, \eqref{two conditions 1} holds for all
$i=1,\dots,n$.
Plugging this into the relation $\varphi_{en}(x)=\tau$ gives 
\begin{equation}\label{tau equal 1}
\begin{aligned}
\tau & = \alpha\sum_{i} \frac{(z_i - \lambda \alpha)_+}{1+\lambda(1-\alpha)} 
+ \frac{(1-\alpha)}{2} \sum_i \frac{(z_i - \lambda \alpha)_+^2}{(1+\lambda(1-\alpha))^2}\\
& = \frac{\alpha}{1+\lambda(1-\alpha)}\sum_{i} (z_i - \lambda \alpha)_+ 
+  \frac{(1-\alpha)}{2(1+\lambda(1-\alpha))^2} \sum_i (z_i - \lambda \alpha)_+^2.
\end{aligned} 
\end{equation}
The strong convexity of the objective in $x$ implies that there is a
unique positive $\lam$ that solves this equation. In addition, for
$\lam \ge \alf^{-1}\inorm{z}$, the right-hand side of \eqref{tau equal
  1} is zero, while for $\lam=0$, the right-hand side is
$\varphi(z)>\tau$.  So the unique optimal $ \lam$ resides in the open
interval $(0,\alf^{-1}\inorm{z})$.  Finally, since
$(1+\lambda(1-\alpha))>0$ for all $\lam\ge 0$, equation \eqref{tau
  equal 1} is equivalent to
\[
0=\tau(1+\lambda(1-\alpha))^2
- \alpha(1+\lambda(1-\alpha))\sum_{i} (z_i - \lambda \alpha)_+ 
- \frac{(1-\alpha)}{2} \sum_i (z_i - \lambda \alpha)_+^2.
\]
The root $\lambda$ is found by sorting coordinates of $z$ and then solving a
quadratic polynomial in $\lam$. Substituting $\lambda$ back into \eqref{two conditions 1}, we find the optimal $x$. 

\subsubsection*{Affine minorant oracle for the value function}

Following the approach of \S\ref{sec:lower minorants via duality}, for
each candidate value of $\tau$ in Algorithm \ref{algo:newton}, we
generate a dual certificate $y$ that yields a lower-bound on the value
function $v(\tau)$. Such dual iterates are generated automatically by
fast gradient methods on the primal problem~\cite{tseng_f_order}. To
obtain an affine minorant of $v$, we then need a method for evaluating
the function
\[
\Phi(y,\tau) :=\ip{b}{y}-\half\tnorm{y}^2-\del^{\star}_{[\vpen\le\tau]}(A^Ty),
\]
and a subgradient $s\in \partial_\tau \del^{\star}_{[\vpen\le\tau]}(A^Ty)$.
To this end, we use the representation
\[
\del^{\star}_{[\vpen\le \tau]}(z)=\inf_{\mu>0}[\tau\mu+\mu\varphi_{en}^{\star}(\mu^{-1}z)].
\]
See, for example, Aravkin et al.~\cite[Equation
6.5c]{AravkinBurkeFriedlander:2013}.  Since $\vpen$ is the sum of two
finite-valued convex functions, its conjugate is the infimal
convolution
\[
\vpen^{\star}(z)=\inf_{v\in\alf\bB_\infty}\frac{1}{2(1-\alf)}\tnorm{z-v}^2
=\frac{1}{2(1-\alf)}\di^2_{\alpha \bB_{\infty}}(z)=\frac{1}{2(1-\alf)}\tnorm{(\abs{z}-\alf e)_+}^2.
\]
 Hence, for $\mu>0$, we have
\begin{equation}\label{opt mu for ren}
\del^{\star}_{[\vpen\le\tau]}(z)=\inf_{\mu>0}~\left\{\tau\mu+\frac{1}{2(1-\alf)\mu}\tnorm{(\abs{z}-\mu \alf\e)_+}^2\right\},
\end{equation}
and the derivative of $\del^{\star}_{[\vpen\le\tau]}(z)$ with respect to $\tau$ is given by the optimal
$\mu$ when it exists.
Note that if $\mu\ge\alf^{-1}\inorm{z}$, then 
$[\tau\mu+\frac{1}{2(1-\alf)\mu}\tnorm{(\abs{z}-\mu \alf\e)_+}^2]=\tau\mu$, while for $\mu\to 0$
we have $[\tau\mu+\frac{1}{2(1-\alf)\mu}\tnorm{(\abs{z}-\alf\mu \e)_+}^2]\to +\infty$. Hence,
an optimal $\mu$ exists when $\tau>0$. It is also unique due to the convex piecewise quadratic
nature of the objective. Consequently, the optimal $\mu$ in \eqref{opt mu for ren} can be obtained by 
sorting $|z|$ and then writing in closed form the solution
of a sequence of elementary univariate convex functions over an interval.

\appendix
\section{Proofs}
\label{appendix:proofs}

\begin{theorem}[Superlinear convergence of Newton and secant methods]
\label{Superlinear}
Let $f\colon \Rn\to\bR$ be a non-increasing, convex function on the
interval $[a,b]$. Suppose that the point $\tau_*:=\inf\{\tau: f(\tau)\leq 0\}$
lies in $(a,b)$ and the non-degeneracy condition
$g_*:=\inf\set{g}{g\in\partial f(\tau_*)}<0$ holds. Fix two points
$\tau_{-1}, \tau_1\in (a,b)$ satisfying $\tau_0<\tau_1< \tau_*$ and consider the
following two iterations:
\begin{equation}
\label{Newton}
\tag{Newton}
\tau_{k+1}:=
\begin{cases}
\tau_k& \mbox{if $f(\tau_k)=0$,}\\
\tau_k-\frac{f(\tau_k)}{g_k}\qquad \mbox{[for $g_k\in\partial f(\tau_k)$]} & \mbox{otherwise;}
\end{cases}
\end{equation}
and
\begin{equation}
\label{Secant}
\tag{Secant}
\tau_{k+1}:=
\begin{cases}
\tau_k&\mbox{if $f(\tau_k)=0$,}\\
\tau_k-\frac{\tau_k-\tau_{k-1}}{f(\tau_k)-f(\tau_{k-1})}f(\tau_k)&\mbox{otherwise.}
\end{cases}
\quad
\end{equation}
If either sequence terminates finitely at some $\tau_k$, then it must be the case $\tau_k=\tau_*$.
If the sequence $\{\tau_k\}$ does not terminate finitely, then
$|\tau_*-\tau_{k+1}|\le (1-\frac{g_*}{\gam_k})|\tau_* -\tau_k|,\ k=1,2,\dots$,
where $\gam_k=g_k$ for the Newton sequence and 
$\gam_k$ is any element of $\partial f(\tau_{k-1})$ for the secant sequence.
In either case, $\gam_k\uparrow g_*$
and $\tau_k\uparrow \tau_*$ globally $q$-superlinearly.
\end{theorem}

\begin{proof}
Since $f$ is convex, the subdifferential $\partial f(\tau)$ is nonempty for all $\tau\in(a,b)$. The claim concerning finite termination is easy to deduce from convexity; we leave the details to the reader.
Suppose neither sequence terminates finitely at $\tau_*$.
Let us first consider the Newton iteration.
Convexity of $f$ immediately implies that the sequence $\tau_i$ is well-defined and satisfies $\tau_0<\tau_1<\tau_2<\dots< \tau_*$. Monotonicity of the subdifferential then implies $g_0\le g_1\le g_2\le \dots \le  g_*<0$.
Due to the inequalities $f(\tau_*)+\bar g(\tau_k-\tau_*)\le f(\tau_k)$ and $g_k<0$, we have
\[
\frac{f(\tau_k)-f(\tau_*)}{g_k}\le -\frac{g_*}{g_k}(\tau_*-\tau_k),
\]
and so
\[
0<\tau_*-\tau_{k+1}= \tau_*-\tau_k+\frac{f(\tau_k)-f(\tau_*)}{g_k}
\le \left(1-\frac{g_*}{g_k}\right)(\tau_*-\tau_k).
\]
Upper semi-continuity of $\partial f$ on its domain implies 
$g_k\uparrow g_*$. 
Hence $\tau_k$ converge $q$-superlinearly to $\tau_*$.

Now consider the secant iteration. As in the Newton iteration, it is immediate from convexity that the sequence $\tau_i$ is well-defined and satisfies $\tau_0<\tau_1<\tau_2<\dots<\tau_*$. Monotonicity of the subdifferential then implies $g_0\le g_1\le g_2\le \dots \le  g_*<0$.
We have \[0<  g_*(\tau_k- \tau_*)\le f(\tau_k)-f(\tau_*),\] and
$f(\tau_{k-1})+g_{k-1}(\tau_k-\tau_{k-1})\le f(\tau_k)$, 
and hence
\[
\frac{\tau_k-\tau_{k-1}}{f(\tau_k)-f(\tau_{k-1})}(f(\tau_k)-f(\tau_*))
\le\frac{f(\tau_k)-f(\tau_*)}{g_{k-1}}<0.
\]
Combining the two inequalities yields
\[
\frac{f(\tau_k)-f(\tau_*)}{f(\tau_k)-f(\tau_{k-1})}(\tau_k-\tau_{k-1})
\le\frac{f(\tau_k)-f(\tau_*)}{g_{k-1}}\le \frac{g_*}{g_{k-1}}(\tau_k-\tau_*)<0.
\]
Consequently, we deduce
\[
0<\tau_* -\tau_{k+1}=\tau_*-\tau_k+
\frac{f(\tau_k)-f(\tau_*)}{f(\tau_k)-f(\tau_{k-1})}(\tau_k-\tau_{k-1})
\le \left(1-\frac{g_*}{g_{k-1}}\right)(\tau_*-\tau_k).
\] The result follows.
\end{proof}

\begin{proof}[Proof of Theorem~\ref{thm:lin_conv_inex_imp}]
	It is easy to see by convexity that the iterates $\tau_k$ are strictly increasing and satisfy $f(\tau_k) >0$.
	For each index $j \geq 2$, define the following quantities:
	$$h_{j}:=\tau_j-\tau_{j-1},\quad \quad \theta_j:=\frac{s_{j}}{s_{j-1}},\quad \textrm{ and } \quad\gamma_j:=\frac{\ell_j}{\ell_{j-1}}.$$
	Note that using the equation $\tau_{j-1}-\tau_j = \frac{\ell_{j-1}}{s_{j-1}}$, we can write $\theta_j=\frac{u_{j-1} - \ell_j}{\ell_{j-1}}$.
	Clearly then the bound, $0\leq \theta_j\leq \alpha-\gamma_j$, is valid.
        Define now constants $\beta_j\in [0,1]$ by the equation $\gamma_j = \beta_j \alpha$.
        Suppose $k \geq 2$ is an index at which the algorithm has not terminated, i.e., $u_k > \epsilon$. 
	Taking into account the inequality $\ell_k \geq \frac{u_k}{\alpha}> \frac{\epsilon}{\alpha}$, we deduce
	\begin{equation} \label{eqn:esp_est1}
	\frac{\epsilon}{\alpha} \leq \ell_k = \ell_1 \prod_{j=2}^k \gamma_j \leq C \alpha^{k-1} \prod_{j=2}^k \beta_j .\
	\end{equation}
	The defining equation for $\tau_{k+1}$ and the definition of $\theta_j$ yield the equality
	\[ h_{k+1} = \frac{\ell_k}{|s_k|} = \frac{\ell_k}{|s_1|} \cdot\prod_{j=2}^k \theta_j^{-1}. \]
	The bounds $\tau_* - \tau_1 \geq h_{k+1}$, $\ell_k \geq \frac{\epsilon}{\alpha}$, and $\theta_j \leq \alpha - \gamma_j$ imply
	\[ \tau_*-\tau_1 \geq \frac{\ell_k}{|s_1|} \cdot\prod_{j=2}^k \theta_j^{-1} \geq \frac{\epsilon}{\alpha |s_1|} (\alpha^{-1})^{k-1} \prod_{j=2}^k (1-\beta_j)^{-1},\]
	and rearranging gives
	\begin{equation} \label{eqn:eps_est2}
	\epsilon \leq (\tau_*-\tau_1) | s_1 | \alpha^k \prod_{j=2}^k (1 - \beta_j) \leq C \alpha^k \prod_{j=2}^k (1 - \beta_j).
	\end{equation}
	Combining \eqref{eqn:esp_est1} and \eqref{eqn:eps_est2}, we get
	\begin{equation}\label{eqn:key_eq} 
	\epsilon \leq C \alpha^k \min \left\{ \prod_{j=2}^k \beta_j, \ \prod_{j=2}^k (1 - \beta_j) \right\}. 
	\end{equation}
	One the other hand, observe 
	$$\left(\prod_{j=2}^k \beta_j\right) \left(\prod_{j=2}^k ( 1- \beta_j)\right)= \prod_{j=2}^k \beta_j( 1- \beta_j) \leq 0.5^{2(k-1)},$$
	and hence
	\begin{equation} \label{eqn:product_est}
	\min \left\{ \prod_{j=2}^k \beta_j, \ \prod_{j=2}^k (1 - \beta_j) \right\} \leq 0.5^{k-1}.
	\end{equation}
	Combining equations \eqref{eqn:product_est} and \eqref{eqn:key_eq}, the claimed estimate
	$ k-1 \leq \log_{2/\alpha} \left( \frac{ \alpha C }{ \epsilon } \right) $
	follows.
	\end{proof}

\begin{proof}[Proof of Theorem~\ref{thm:conv_inex_newt}]
  The proof is identical to the proof of
  Theorem~\ref{thm:lin_conv_inex_imp}, except for some minor
  modifications. The
  only nontrivial change is how we arrive at the bound
  $\theta_j \leq \alpha - \gamma_j$.  For this, observe
  $\tau_{j-1} - \tau_j = \ell_{j-1}/s_{j-1}$, and because the function
  $\tau \mapsto \ell_j + s_j ( \tau - \tau_j)$ minorizes $v$, we see
	\begin{align*}
	u_{j-1} &\geq \ell_j + s_j( \tau_{j-1} - \tau_j ) = \ell_j + s_j \inpar{ \frac{\ell_{j-1}}{s_{j-1}} } = \ell_j + \theta_j \ell_{j-1}.
	\end{align*}
	After rearranging, we get the desired upper bound on $\theta_j$:
	\begin{align*}
	\theta_j &\leq \frac{u_{j-1} - \ell_j}{\ell_{j-1}} \leq \alpha - \gamma_j.
	\end{align*}
Finally, we remark that with the approximate Newton method, we can start indexing at $j = 0$ instead of $j = 1$.
This explains the different constants in the convergence result.
\end{proof}

\begin{lemma}[Concavity of the parametric support function]
  \label{lem:conc_supp}
  For any convex function $f\colon\Rn\to\overline{\bR}$ and vector
  $z\in\bR^n$, the univariate function
  $t\mapsto \delta^*_{[f\leq t]}(z)$ is concave.
\end{lemma}
\begin{proof}
  Convexity of $f$ immediately yields the inclusion
  $$\lambda \cdot[f\leq a]+(1-\lambda) \cdot[f\leq b]\subseteq [f\leq
  \lambda a+(1-\lambda) b] \qquad\forall a,b\in\bR
  \text{ and } \lambda\in [0,1].$$
  We deduce
  $\lambda\cdot\delta^*_{[f\leq a]}(z)+(1-\lambda)
  \cdot\delta^*_{[f\leq b]}(z)=\delta^*_{\lambda \cdot[f\leq
    a]+(1-\lambda) \cdot[f\leq b]}(z)\leq \delta^*_{[f\leq \lambda
    a+(1-\lambda) b]}(z)$, and the result follows.
\end{proof}

\begin{proof}[Proof of Proposition~\ref{cl:ineq_quad}]
  For this proof only, let $\|\cdot\|$ denote the 2-norm.  Note the
  inclusion
  $s/\|y\|\in \partial_{\tau}\Phi_1\left(y/\|y\|,\,\tau\right)$. Use
  the same computation from~\eqref{eqn:lowe-min} to deduce that the
  affine function
  \[\tau'\mapsto (\hat\ell-\sigma)-\frac{s}{\|y\|}(\tau'-\tau)\]
  minorizes $f_1$.

  From the definition of $\hat\ell$, $\Phi_1$, and $\Phi_2$, it
  follows that
  \begin{equation}\label{eqn:equiv_pos}
    \begin{aligned}
        \frac{u-\sigma}{\hat{\ell}-\sigma}
      = \frac{(u-\sigma)\|y\|}{\Phi_2(y,\tau) + \frac{1}{2}\|y\|^2 - \sigma\|y\|} 
      = \frac{2(u-\sigma)\|y\|}{\ell^2 + \|y\|^2 - 2\sigma\|y\|}.
    \end{aligned}
  \end{equation}
  Taking into account the equivalence
  \[
    \frac{u-\sigma}{\ell-\sigma} \leq \alpha
    \quad\iff\quad
    \frac{u + (\alpha-1)\sigma}{\alpha} \leq \ell,
  \]
  we deduce
  \[
    \ell^2+\|y\|^2-2\sigma\|y\|
    \geq
    \alpha^{-2}
    \Big(
      (u+(\alpha-1)\sigma)^2+\|\alpha y\|^2-2\sigma\alpha\|\alpha y\|
    \Big)
    \geq 2\alpha^{-1}(u-\sigma)\|y\|,
  \]
  where the rightmost inequality follows from the computation
  \[
    \begin{aligned}
      (u + [\alpha- 1]\sigma)^2 &+ \|\alpha y\|^2 -
      2\alpha\sigma\|\alpha y\| - 2(u-\sigma)\|\alpha y\|
    \\& = \left(u + [\alpha-1]\sigma\right)^2 + \|\alpha y\|^2 -
      2\|\alpha y\|(u + [\alpha-1]\sigma)
    \\&= \left(u + [\alpha-1]\sigma - \|\alpha y\|\right)^2  \geq 0.
    \end{aligned}
  \]
  Because the right-hand side of~\eqref{eqn:equiv_pos} is
  non-negative, we can deduce that $\hat{\ell}\geq \sigma$. Finally,
  the required inequality $(u-\sigma)/(\hat{\ell}-\sigma)\leq \alpha$
  also follows from~\eqref{eqn:equiv_pos}.
\end{proof}

\begin{lemma}
	$(-\lambda\submin)^{\star}(y) = \delta_\cS(-y)$, where
	$\cS= \cK^{*} \cap \left\{ x \mid \ip e x = 1 \right\}$.
	\label{rem:conjugate_lamdbdaMin}
\end{lemma}
\begin{proof}
The following formula is established in \cite{ren_conic}:
\[
  \partial (-\lambda_{\min}) ( x )
  =
  \left\{ -y \mid \ip{y}{e} = 1, \
    \ip{ y}{ z - (x - \lambda_{\min}(x) e ) } \geq 0 \text{ for
      all } z \in \mathcal{K} \right\}
\]
or equivalently
\begin{align*}
\partial (-\lambda_{\min}) ( x ) 
&= \set{ -y }{ \ip{y}{e} = 1,  -y\in N_{\mathcal{K}}\left(x - \lambda_{\min}(x) e\right) }\\
&= \set{ -y}{ \ip{y}{e} = 1, \ y \in \mathcal{K}^*, \  0 = \lambda_{\min}(x) - \ip{y}{x} }.
\end{align*}
Here the symbol $N_{\mathcal{K}}$ denotes the normal cone to $\mathcal{K}$.
Now for any $y\in \partial (-\lambda_{\min}) ( x )$, we have 
$\langle x,y\rangle=-\lambda_{\min}(x)$. Observe $\range \partial (-\lambda_{\min})=-\mathcal{S}$. Hence by the equality in the Fenchel-Young inequality, for any $y\in -\mathcal{S}$, we have $(-\lambda_{\min})^{\star}(y)=0$.
On the other hand, for any $y$ with $\langle y,e\rangle\neq  -1$, we have
$(-\lambda_{\min})^{\star}(y)\geq \langle te,y\rangle -(-\lambda_{\min})(te)=t(\langle y,e\rangle+1)$ for any $t \geq 0$. Letting $t\to \infty$, we deduce $(-\lambda_{\min})^{\star}(y)=+\infty$. Similarly, consider $y\notin -\mathcal{K}^*$. Then we may find some $x\in \mathcal{K}$ satisfying $\langle x,y\rangle>0$. 
We deduce $(-\lambda_{\min})^{\star}(y)\geq \langle tx,y\rangle -(-\lambda_{\min})(tx)=t(\langle y,x\rangle-(-\lambda_{\min})(x))$ for any $t \geq 0$. Letting $t\to \infty$, we deduce $(-\lambda_{\min})^{\star}(y)=+\infty$. We deduce that $(-\lambda_{\min})^{\star}$ is the indicator function of $-\mathcal{S}$, as claimed.
\end{proof}

\begin{lemma} \label{lem:sum-gauges}
  Let $\cD_1$ and $\cD_2$ be two nonempty closed convex sets that
  contain the origin. Then
  $\gamma_{\cD_1} + \gamma_{\cD_2} =
  \gamma_{(\cD_1^\circ+\cD_2^\circ)^\circ}$.
  If additionally $0\in\interior{\cD_1}$, then
  $(\gamma_{\cD_1} + \gamma_{\cD_2})^\circ =
  \gamma_{\cD_1^\circ+\cD_2^\circ}.$
\end{lemma}
\begin{proof}
  Theorem~14.5 of \cite{RTR} contains most of the needed tools. In
  particular, the gauge of any closed convex function containing the
  origin is the support function of the polar. Thus,
  \[
    \gamma_{\cD_1} + \gamma_{\cD_2}
    = \delta^*_{\cD_1^\circ} + \delta^*_{\cD_2^\circ}.
  \]
  By \cite[Cor.~3.2.5]{hiriart-urruty01}, we have
  \[
    \delta^*_{\cD_1^\circ} + \delta^*_{\cD_2^\circ}
    = \delta^*_{\cl{\cD_1^\circ+\cD_2^\circ}}
    = \delta^*_{\cD_1^\circ+\cD_2^\circ},
  \]
  where the last equality holds because the support function does not
  distinguish a set from its closure. Again using the polarity
  correspondence between the gauge and support functions, we have
  $\gamma_{\cD_1} + \gamma_{\cD_2} =
  \gamma_{(\cD_1^\circ+\cD_2^\circ)^\circ}$,
  as required. We now prove the second part of the lemma. Use the
  first part of the result and \cite[Thm.~15.1]{RTR} to deduce that
  \begin{equation}\label{eq:1}
    (\gamma_{\cD_1}+\gamma_{\cD_2})^\circ
    = \gamma_{(\cD_1^\circ+\cD_2^\circ)^{\circ\circ}}.
  \end{equation}
  Because $0\in\interior{\cD_1}$, the set $\cD_1^\circ$ is compact
  \cite[Cor.~14.5]{RTR}, and because $\cD_2^\circ$ is closed,
  $\cD_1^\circ+\cD_2^\circ$ is also closed. Thus,
  $(\cD_1^\circ+\cD_2^\circ)^{\circ\circ}=\cD_1^\circ+\cD_2^\circ$. It
  then follows from~\eqref{eq:1} that
  $(\gamma_{\cD_1}+\gamma_{\cD_2})^\circ=\gamma_{(\cD_1^\circ+\cD_2^\circ)}$,
  as required.
\end{proof}

\begin{rem}[Projection onto a conic slice sets]\label{rem:proj_base}
This remark is standard. Fix a proper convex cone $\cK$ and consider the projection problem
\[
  \min_{x}\ \set{\half\|x-z\|^2}{\langle c,x\rangle=1,\, x\in
  \mathcal{K}}.
\]
Equivalently, we can consider the univariate concave maximization problem 
\begin{align*}
\max_{\beta}\min_{x\in \cK}\, L(x,\beta)&= \max_{\beta}\min_{x\in \cK}\,  \half\|x-z\|^2+\beta(\langle c,x\rangle-1)\\
&=\max_{\beta}\min_{x\in \cK}\, \half\|x-(z-\beta c)\|^2+\beta(\langle c,z\rangle -1)-\half\beta^2\|c\|^2\\
&=\max_{\beta}~ \half\di^2_{\mathcal{K}}(z-\beta c)+\beta(\langle c,z\rangle -1)-\half\beta^2\|c\|^2. 
\end{align*}
We can solve this problem for example by bisection, provided projections onto $\cK$ are available.
\end{rem}

\bibliographystyle{abbrvnat}
\bibliography{dima_bib,references_sasha}

\end{document}